\documentclass[12pt,a4paper, final]{amsart}

\usepackage[a4paper,margin=1in]{geometry}

\usepackage{amsmath,amssymb, verbatim}
\usepackage{comment}
\usepackage{amsthm}
\usepackage{tikz-cd}
\usepackage{ifthen}
\usetikzlibrary{babel}

\newtheorem{theorem}{Theorem}[section]
\newtheorem{lemma}[theorem]{Lemma}
\newtheorem{corollary}[theorem]{Corollary}
\newtheorem{conjecture}[theorem]{Conjecture}

\theoremstyle{definition}
\newtheorem{definition}[theorem]{Definition}
\newtheorem{proposition}[theorem]{Proposition}
\newtheorem{remark}[theorem]{Remark}

\renewcommand{\P}{{\mb P}}

\def\Q{\mathbb Q}
\def\wdw{\wedge\dots\wedge}

\def\L{\Lambda}
\def\oL{\ol L}
\def\t{\times}
\def\k{K}
\def\cl{\colon}

\def\cl{\colon}
\def\emb{\hookrightarrow}
\def\A{{\mb A}}
\def\ol{\overline}

\def\F{\mathcal F}
\def\mc{\mathcal}
\def\wt{\widetilde}

\def\ph{\varphi}
\def\mb{\mathbb}
\def\wt{\widetilde}

\def\bs{\backslash}

\def\A{\mb A}
\def\Z{\mb Z}
\def\Q{\mb Q}
\def\Set{\textbf{Set}}
\def\F{\mathbf{Fields}}
\newcommand{\res}[2]{\left.{#1}\right|_{#2}}

\newcommand{\SRL}{\mathrm{RecMaps}}
\newcommand{\Hom}{\mathrm{Hom}}
\def\B{{\mc B}}

\DeclareMathOperator{\ts}{\partial}
\DeclareMathOperator{\ord}{ord}
\DeclareMathOperator{\mult}{mult}
\DeclareMathOperator{\val}{val}

\DeclareMathOperator{\im}{Im}
\DeclareMathOperator*{\colim}{colim}

\DeclareMathOperator{\spec}{Spec}
\DeclareMathOperator{\dval}{DiscVal}

\DeclareMathOperator{\Div}{Div}
\def\N{\mc N}
\def\ard{\dashrightarrow}

\begin{document}
\title[Higher Chow groups and not necessarily admissible cycles]{Higher Chow groups and not necessarily admissible cycles}
\author{Vasily Bolbachan}
\email{vbolbachan@gmail.com}
\address{Skolkovo Institute of Science and Technology, Nobelya str., 3, Moscow, Russia, 121205; Faculty of Mathematics, National Research University Higher School of Ecnomics, Russian Federation, Usacheva str., 6, Moscow, 119048; HSE-Skoltech International Laboratory of Representation
Theory and Mathematical Physics, Usacheva str., 6, Moscow, 119048}


\begin{abstract}
    We construct some analog of cubical Bloch's higher Chow groups. Instead of considering cycles in $X\t\A^n$ we consider varieties $Y$ over $X$ together with a distinguished element in the $n$-th exterior power of the multiplicative group of the field of fraction on $Y$. This definition allows us to make sense of a cycle in $X\t\A^n$ intersecting faces improperly as an element in this complex. 

    We prove that this complex is well-defined and study its basic properties: flat pullback, the localization sequence etc.  As an application we prove that the cohomology of this complex in degrees $m-1, m$ and weight $m$ isomorphic to the cohomology of polylogarithmic complex.
\end{abstract}

\thanks{This paper is an output of a research project implemented as part of the Basic Research Program at the National Research University Higher School of Economics (HSE University). This paper was also supported in part by the contest "Young Russian Mathematics"}
\maketitle

\section{Introduction}
\label{sec:introduction}
Everywhere we work over $\Q$. So any abelian group is supposed to be tensored by $\Q$. Fix some field $\k$ of characteristic zero. All varieties are defined over $\k$.

Let $X$ be an algebraic scheme over $\k$. The definition of Bloch's higher Chow group of $X$ can be found in \cite{bloch1986algebraic}. We use cubical version \cite{levine1994bloch}. This group is defined as cohomology of certain complex $z^r(X, 2r-*)$. The group $z^r(X, n)$ is a quotient of the group of codimension $r$ algebraic cycles on $X\t (\P^1\bs\{1\})^n\cong X\t \A^n$ intersecting all faces properly by the subgroup of degenerate cycles. The differential is given by the signed sum of intersections with codimension one faces.

Consider the simplest case when $X$ is a point. In this case the data of an irreducible cycle on $(\P^1\bs\{1\})^n$ is the same as the data of an algebraic variety $Z$ together with $n$ rational functions $\xi_1,\dots, \xi_n$ on $Z$ satisfying certain conditions. The functions $\xi_i$ are just restrictions of coordinate functions to $Z$.

The main idea of this paper is to replace the data $(Z, \xi_1, \dots, \xi_n)$ by the data $(Z, w)$, where $w$ is an arbitrary element in $\L^n(\k(Z)^\t)$. Here $\k(Z)^\t$ is the multiplicative group of the field of rational functions on $Z$ considered as an abelian group.  The element $(Z, \xi_1,\dots, \xi_n)$ corresponds to $(Z, \xi_1\wdw \xi_n)$. The differential is given by so-called tame symbol map defined by A. Goncharov \cite{goncharov1995geometry}.

In particular, any irreducible cycle $Z\subset (\P^1\bs\{1\})^n$ not lying in the union of faces gives an element in this complex. This cycle may intersect faces of codimension two and higher improperly. This is what makes our complex different from classical higher Chow group.

Let us make this idea precise. We recall that \emph{an alteration} is a proper surjective morphism which is generically finite.

    \begin{definition} 
    Let $X$ be an algebraic scheme over $\k$ and $m, j$ be integers. Set $p=\dim X+m-j$ and $n=2m-j$. Denote by $\wt\L(X, m)_j$  a vector space freely generated by the isomorphism classes of the triples $(Y, a, f)$, where: $Y$ is a variety over $\k$ of dimension $p$, $f\colon Y\to X$ is a proper morphism and $a\in\Lambda^n(\k(Y)^\t)$. Denote by $[Y,a,f]\in\wt\L(X,m)_j$ the corresponding element.
Denote by $\L(X, m)_j$ the quotient of $\wt\L(X, m)_j$ by the following relations:
$$[\wt Y,\ph^*(a),f\circ \ph]=(\deg\ph)[Y,a,f].$$
$$[Y,a+b,f]=[Y, a, f] + [Y, b, f].$$
$$[Y,\lambda\cdot a,f]=\lambda[Y, a, f].$$

In this formula $\ph\colon \wt Y\to Y$ is any alteration and $\lambda\in \mathbb Q$. The map $\ph^*$ is defined by the formula
$\ph^*(\alpha_1\wdw \alpha_n)=\ph^*(\alpha_1)\wdw \ph^*(\alpha_n)$. We remark that we write group law in the group $\Lambda^n\k(Y)^\t$ additively. So in the case $n=1$ addition in this group corresponds to \emph{multiplication} of functions.

\end{definition}

Below we will give a precise definition of the differential in this complex and will show that in this way we get a well-defined complex. The fact that $d^2=0$ will follow from some version of Parshin reciprocity law \cite{parshin1975class,bolbachan_2023_chow}.

\subsection{The statement of the main results}
We will prove the following theorem:

\begin{theorem}
    The complex $\Lambda(X,m)$ is well-defined. It satisfies the following properties:
    \begin{enumerate}
        \item (Functoriality) The complex $\Lambda(\cdot, m)$ is covariantly functorial for proper maps and contravariantly functorial for flat maps.

        \item (Localization) Let $Z$ be a closed  subset of $X$ of codimension $k$. Denote by $i$ the embedding $Z\emb X$ and by $j$ the complementary open embedding $U\emb X$. We have the following exact sequence of complexes:
        $$0\to \L(Z, m-k)[-2k]\xrightarrow{i_*}\L(X,m)\xrightarrow{j^*}\L(U,m)\to 0.$$
        \item Denote by $z^m(X, 2m-*)$ the complex calculating cubical Bloch's higher Chow group in weight $m$. There is the canonical morphism of complexes $$z^m(X, 2m-*)\to \Lambda(X, m)_*.$$
    \end{enumerate}
\end{theorem}

\begin{conjecture}
    The map $z^m(X,2r-*)\to \L(X,m)_*$ is a quasi-isomorphism. 
\end{conjecture}

It is interesting that in our case, the localization sequence holds on the level of complexes. Compare our proof with nontrivial arguments from \cite{bloch_1994_moving_lemmas}.

We recall that $\k$ is the base field. 
The following statement is of independent interest.
\begin{theorem}
\label{th:Lambda_generators_intro}
   The group $\Lambda(\k, m)/\im(d)$ is generated by the elements of the form $[Y'\t \P^1, a, f].$
\end{theorem}

Even in the case $j=m-1$, the corresponding statement for cubical Bloch's higher Chow group is not known (see \cite{gerdes1991linearization} for the simplicial case). 

The second part of this paper is devoted to the connection of the complex $\L(\k, m)$ with polylogarithmic complex defined by A. Goncharov in \cite{goncharov1994polylogarithms, goncharov1995geometry}. Denote weight $m$ polylogarithmic complex of the field $\k$ by $\Gamma(\k,m)$. 

Let $C$ be a complex. Denote by $\tau_{\geq k}C$ its canonical truncation. Define a morphism of complexes $$\mc T_{\ge m-1}(m)\colon \tau_{\ge m-1}\Gamma(\k,m)\to \tau_{\ge m-1}\Lambda(\k, m)$$ as follows. The element $\{a\}_2\wedge c_3\wdw c_m$ goes to
$$[\P^1, t\wedge (1-t)\wedge (1-a/t)\wedge c_3\wdw c_m].$$
The element $c_1\wdw c_m$ goes to
$[\spec \k, c_1\wdw c_m]$.

The construction of this map goes back to Totaro \cite{totaro_1992_milnork}. See also \cite{bloch_rriz_1994_mixed,goncharov_levin_gangl_2009_muly_polyl_alg_cycles,gangl_herbert_poly_identities_chow_group,zhao2007goncharov}.
Here is our main result:

\begin{theorem}    \label{th:Tatoro_is_isomorphism_intro}
    Let $\k$ be a field of characteristic zero. The map $\mc T_{\ge m-1}(m)$  is a quasi-isomorphism. In particular for $j=m-1, m$ we have $$H^j(\Lambda(\spec \k, m))\cong H^j(\Gamma(\k, m)).$$
    Moreover, when $\k$ is algebraically closed the map $\mc T_{\ge m-1}(m)$ is an isomorphism of complexes.
\end{theorem}

It follows from the proof of this theorem, that in the case $m=2$ the map $\mc T_{\geq 1}(2)$ can be lifted to a map $\mc T(2)\colon \Gamma(\k, 2)\to \L(\k, 2)$. It is surprising, because in the case of higher Chow group there is no natural map between the corresponding complexes. The reason is that Abel's five-term relation holds only modulo coboundary.

\subsection{Preliminary definitions}
\label{sub:sec:preliminary_def}
We recall that we work over $\mathbb Q$. Let $\k$ be a field. \emph{An algebraic scheme over $\k$} is a finite type scheme over $\k$. \emph{A variety} is an algebraic scheme which is integral and separated. Throughout the article we assume that the field $\k$ has characteristic zero. All schemes are assumed to be equidimensional. \emph{An alteration} $\ph\colon \wt Y\to Y$ is a proper surjective morphism between integral schemes which is generically finite. Unless otherwise stated, we always assume $\wt Y$ to be a smooth variety.

Let $(F,\nu)$ be a discrete valuation field. Denote $\mc O_\nu=\{x\in F|\nu(x)\geq 0\}, m_\nu=\{x\in F|\nu(x)>0\}$ and $\overline F_\nu =\mc O_\nu/m_\nu$. We recall that an element $a\in F^\t$ is called \emph{a uniformiser} if $\nu(a)=1$ and \emph{a unit} if $\nu(a)=0$. For $u\in \mc O_\nu$ denote by $\overline u$ its residue class in $\ol F_\nu$.

The proof of the following proposition can be found in \cite{goncharov1995geometry}:

\begin{proposition}
\label{prop:tame_symbol_classiacal}
Let $(F,\nu)$ be a discrete valuation field and $n\geq 1$. There is a unique map $\ts_\nu\colon \L^n(F^\t)\to \L^{n-1}\ol F_\nu^\t$ satisfying the following conditions:

\begin{enumerate}
    \item When $n=1$ we have $\ts_{\nu}(a)=\nu(a)$.
   \item For any units $u_2,\dots u_n\in F$ we have $\ts_\nu(x\wedge u_2\wdw u_n)=\nu(x)\wedge \ol{u_2} \dots \wedge \overline{u_n}$.

\end{enumerate}
\end{proposition}

The map $\ts_\nu$ is called \emph{tame symbol map}. Throughout the paper we will repeatedly use the following lemma:

\begin{lemma}
\label{lemma:leibniz_rule_tame_symbol}
Let $(F, \nu)$ be a discrete valuation field. Let $k, n$ be two natural numbers satisfying the condition $k<n$. Let $a_1,\dots, a_n\in F^\t$ such that $\nu(a_{k+1}), \dots, \nu(a_n)=0$. Then the following formula holds:
$$\ts_\nu(a_1\wedge\dots\wedge a_n)=\ts_\nu(a_1\wedge\dots\wedge a_k)\wedge \overline{a_{k+1}}\wedge \dots\wedge \overline{a_n}.$$
\end{lemma}

Let $X$ be a smooth algebraic variety and $D\subset X$ be an irreducible divisor. Denote by $\nu_D$ a discrete valuation corresponding to $D$. Abusing notation we write $\ts_D$ for $\ts_{\nu_D}$.

Let $Y$ be a variety. Denote $\L^n(Y):=\L^n(\k(Y)^\t\otimes_\Z\Q)$. Here $\k(Y)^\t$ is the multiplicative group of the field $\k(Y)$ considered as an abelian group.

Let $\ph\colon Y_1\to Y_2$ be a dominant morphism between varieties. Denote by $\ph^*$ the map $\k(Y_2)^\t\to \k(Y_1)^\t$ given by the formula $\ph^*(\alpha)=\alpha\circ \ph$. We will denote by the same symbol the map $\L^n(Y_2)\to \L^n(Y_1)$ given by the formula $\ph^*(\alpha_1\wdw \alpha_n)=\ph^*(\alpha_1)\wdw \ph^*(\alpha_n)$.

Let $D_1\subset Y_1, D_2\subset Y_2$ be irreducible divisors such that $\varphi(D_1)= D_2$. Denote by $e(D_1, D_2)$ the multiplicity of $\ph^{-1}(D_2)$ along $D_1$. This number can be computed as $\ord_{D_1}(\ph^*(\alpha))$, where $\alpha$ is any rational function on $Y_2$ satisfying $\ord_{D_2}(f)=1$. 

\begin{lemma}
\label{lemma:functoriality_of_residue_2}
    Let $\ph\colon Y_1\to Y_2$ be a dominant morphism between smooth varieties. Let $D_1\subset Y_1, D_2\subset Y_2$ be two divisors such that $\ol{\ph(D_1)}=D_2$. Denote by $\ol\ph $ the natural map $D_1\to D_2$. Then for any $a\in\L^n(Y)$ we have:
    $$\ts_{D_1}(\ph^*(a))=e(D_1, D_2)(\ol\ph)^*(\ts_{D_2}(a)).$$
\end{lemma}
\begin{proof}
    We can assume that $a=\xi_1\wdw \xi_n$ such that $\ord_{D_1}{\xi_1}=1$ and $\ord_{D_1}(\xi_i)=0$ for any $i>1$. The number $\ord_{D_1}(\ph^*(\xi_i))$ is equal to $e(D_1, D_2)$ if $i=1$ and is equal to $0$ otherwise. Now the statement follows from Lemma \ref{lemma:leibniz_rule_tame_symbol}.
\end{proof}

Let us assume that $\ph$ is a morphism of finite degree. We have a finite extension $\k(D_1)\subset \k(D_2)$. Denote its degree by $f(D_1, D_2)$. Denote by $Div(\ph)_0$ the set of divisors contracted under $\ph$. Let $D\subset Y_2$ be an irreducible divisor. Denote by $Div(\ph, D)$ the set of irreducible divisors $D'\subset Y_1$ such that $\ph(D')=D$. There is a bijection between the set $Div(\ph, D)$ and the set of extensions of the discrete valuation $\nu_D$ to the field $\k(Y_1)$. The following statement is well known:

\begin{lemma}
\label{lemma:degree_is_a_sum_multiplicties}
    Let $\ph\colon Y_1\to Y_2$ be an alteration. Let $D\subset Y_2$ be an irreducible divisor. The following formula holds:  
    $$\deg \varphi = \sum\limits_{D'\in Div(\ph, D)}e(D', D)f(D', D).$$
\end{lemma}

For a divisor $D=\sum\limits_{\alpha}n_\alpha[D_\alpha]$ on $X$, denote by $|D|$ \emph{its support} defined by the formula $\sum\limits_{\alpha} [D_\alpha]$. A divisor $D$ is called \emph{supported on a simple normal crossing divisor} if $|D|$ is a simple normal crossing divisor. Let $x\in X$. A divisor $D$ is called \emph{supported on a simple normal crossing divisor locally at $x$}, if the restriction of this divisor to some open neighborhood of the point $x$ is supported on a simple normal crossing divisor.

\begin{definition}
\label{def:strictly_regular}
    Let $X$ be a smooth variety and $x\in X$. An element $a\in\L(X)^n$ is  called \emph{strictly regular at $x$} if $a$ can be represented  in the form
    $$\sum_\beta n_\beta \alpha_1^{\beta}\wdw \alpha_n^{\beta}$$
    such that the divisor $\sum\limits_{\beta,j}|(\alpha_j^\beta)|$ is supported on a simple normal crossing divisor locally at $x$. 

    The element $a$ is called \emph{strictly regular} if this condition holds for any $x\in X$.
\end{definition}

The definition of Bloch's higher Chow group can be found in \cite{levine1994bloch}.  Denote by $\square^n=(\P^1\bs\{1\})^n$ the algebraic cube of dimension $n$. Subvarieties of $X\t \square^n$ given by the equations of the form $z_i=0,\infty$ are called faces. Let $\wt z_p(X,n)$ be the free abelian group generated by dimension $p$ cycles on $X\t \square^n$ intersecting all faces properly. We have the natural projections $\pi_{i,n}\colon \square^n\to\square^{n-1}$ given by the formula $\pi_{i,n}(t_1,\dots, t_n)=(t_1,\dots\hat{t_i}, \dots, t_n)$. Denote by $D_p(X, n)\subset \wt z_p(X,n)$ a subgroup generated by the images of the maps  $(id\t \pi_{i,n})^*\colon \wt z_{p-1}(X, n-1)\to \wt z_{p}(X, n), i=1,\dots, n$. Finally define $$z_p(X,n)=\wt z_p(X, n)/D_p(X,n).$$ One can define a differential on $z_p(X,n)$ given by the alternating sum of intersections with codimension one faces \cite{levine1994bloch} (we use the convention that the intersection with the face given by the equation $z_1=0$ has positive sign). In this way we get a complex

$$\dots \to z_{p+1}(X, n+1)\to z_p(X,n)\to z_{p-1}(X, n-1)\to\dots$$

It is convenient to pass to motivic grading. As in the case of the complex $\Lambda$, denote this complex by $CH(X, m)$, where $z_p(X,n)$ sits in degree $2\dim X+n-2p$ and $m=\dim X+n-p$.

\subsection{Definition of the complex $\Lambda(X, m)$}
\label{sub:sec:Lambda_definition}
\begin{definition}
    \label{def:Lambda}
    Let $X$ be an algebraic scheme.  Let $m, j\in\Z$. Set $p=\dim X+m-j, n=2m-j$. We have $m=\dim X+n-p$ and $j=2\dim X+n-2p$. Throughout the paper we will use the pairs of integers $(m,j), (n, p)$ interchangeably.
    \end{definition}
    \begin{definition} 
    Denote by $\wt\L(X, m)_j$ a vector space freely generated by  the isomorphism classes of the triples $(Y, a, f)$, where: $Y$ is a variety over $\k$ of dimension $p$, $f\colon Y\to X$ is a proper morphism and $a\in\Lambda^n(\k(Y)^\t\otimes_\Z\Q)$. Denote by $[Y,a,f]\in\wt\L(X,m)_j$ the corresponding element.
Denote by $\L(X, m)_j$ the quotient of $\wt\L(X, m)_j$ by the following relations:
$$[\wt Y,\ph^*(a),f\circ \ph]=(\deg\ph)[Y,a,f].$$
$$[Y,a+b,f]=[Y, a, f] + [Y, b, f].$$
$$[Y,\lambda\cdot a,f]=\lambda[Y, a, f].$$

In this formula $\ph\colon \wt Y\to Y$ is any alteration and $\lambda\in \mathbb Q$. The map $\ph^*$ is defined by the formula
$\ph^*(\alpha_1\wdw \alpha_n)=\ph^*(\alpha_1)\wdw \ph^*(\alpha_n)$. 
\end{definition}
Let us formulate this definition in a more categorical fashion. Denote by ${\mc Var}_{X,p}$ the category of all dimension $p$ proper varieties over $X$ and their proper dominant morphisms. 

Consider the functor $\Lambda^n$ on this category. Its value on some variety $Y\in {\mc Var}_{X,d}$ is equal to the vector space $\L^n(\k(Y)^\t\otimes_{\mb Z}{\mb Q})$. For a morphism $\varphi\colon Y_1\to Y_2$ the corresponding map $\Lambda^n(Y_2)\to \Lambda^n(Y_1)$ is defined by the formula $\Lambda^n(\varphi)(a)=\dfrac 1{\deg{\varphi}}(\ph^*(a))$.

\begin{definition}
    $$\L(X, m)_j=\colim\limits_{Y\in {\mc Var}_{X,p}}\Lambda^n(Y).$$
\end{definition}

As the category ${\mc Var}_{X,d}$ is essentially small this colimit is a well-defined vector space over $\Q$. It is easy to see that the two definitions bellow are equivalent.

For a variety $Y$, a proper morphism $f\colon Y\to X$ and $a\in \Lambda^n(Y)$, we denote the corresponding element in $\L(X, m)_j$ by $[Y, a, f]$. We will abbreviate it to $[Y, a]$ in the case when $X$ is a point or the map $f$ is clear from context. For any proper dominant morphism $\ph\cl Y_1\to Y_2$ we have $[Y_2, a, f]=\dfrac 1{\deg \ph}[Y_1, \ph^*(a), f\circ\ph]$.

Let $X$ be a smooth variety, $f\colon  Y\to X$ be a proper morphism, $a\in \L^n(Y)$ and $D\subset Y$ be an irreducible divisor. Define
$$\ts_D(a, f)=[D, \ts_D(a), \res{f}{D}].$$

Define a map $d\colon \L(X,m)_j\to \L(X,m)_{j+1}$ as follows. Let $[Y, a, f]\in \L(X,m)_j$. Choose an alteration $\ph \colon \wt Y\to Y$ with smooth $\wt Y$ such that the element $\ph^*(a)$ would be strictly regular. Such an alteration exists by results of \cite{deJong1996smoothness}. Define
$$d([Y, a, f])=\dfrac 1{\deg\ph}\sum\limits_{D\subset \wt Y}\ts_D(\ph^*(a), f\circ \ph).$$

In this formula the sum is taken over all irreducible divisors $D\subset \wt Y$. In the next section we will show that this expression does not depend on $\ph$ and is well-defined.

\subsection{The outline of the paper} In Section \ref{sec:Lambda_is_well_defined} we will prove that the complex $\L(X, m)$ is well defined. In Section \ref{sec:Lambda_properties} we will prove some basic properties of this complex: flat pullback, localization sequence etc. Section 4 is devoted to the proof of Theorem \ref{th:Lambda_generators_intro} describing some generators of the complex $\L(X,m)$ in the case when $X$ is a point. In Sections 5 we will prove Theorem \ref{th:Tatoro_is_isomorphism_intro} modulo results from Section 6. Section 6 is devoted to some extension of my main result from \cite{bolbachan_2023_chow}.

\subsection{Acknowledgements}
The author is grateful to A.Levin, D. Rudenko and S. Bloch for useful discussions.

\section{The complex $\L(X,m)$ is well-defined}
\label{sec:Lambda_is_well_defined}

The goal of this section is to prove the following theorem:

\begin{theorem}
\label{th:Lambda(m)_well_defined}
    Let $X$ be an algebraic scheme over $\k$. The complex $\L(X,m)$ is well-defined.
\end{theorem}

\subsection{Some vanishing statements}
\label{sub:sec:vanishing_statements}
The goal of this subsection is to prove the following theorem:
\begin{theorem}
\label{th:differential_strictly_regular_elements}
    Let $Y$ be a smooth variety and $f\colon Y\to X$ be a proper morphism. Let $a\in \L^n(Y)$ be strictly regular. Then for any alteration $\varphi \colon \wt Y\to Y$ and any divisor $E\subset \wt Y$ contracted under $\varphi$ we have $\ts_E(\varphi^*(a), f\circ \ph)=0\in \L(X, m)_{j+1}$. 
\end{theorem}

To prove this theorem we need several lemmas.

\begin{lemma}
    \label{lemma:degenerate_cycles}
    Let $Y$ and $R$ be varieties. Let $f_1\colon Y\to R, f_2\colon R\to X$ be proper morphisms and assume that $f_1$ is dominant. Assume that the corresponding extension $\k(R)\subset \k(Y)$ has positive transcendent degree.  Then for any $a\in\L^n(R)$, we have $$[Y,f_1^*(a), f_2\circ f_1]=0.$$
\end{lemma}

\begin{remark}
    The cycles of this form correspond to so-called degenerate cycles in higher Chow group. So this lemma shows that these elements are zero in $\Lambda(X,m)$ automatically and we do not need to take quotient by these elements.
\end{remark}

\begin{proof}
    Denote by $s$ the transcendent degree of the extension $\k(R)\subset \k(Y)$. Choose a transcendence basis $t_1,\dots, t_s$ of $\k(Y)$ over $\k(R)$. It is easy to see that the extension
    $$\k(R)(t_1,\dots, t_s)\subset \k(Y)$$ is finite. This implies that there is a rational map $g\colon Y\ard R\t(\P^1)^s$ of finite degree such that $f_1=\pi\circ g$, where $\pi\colon R\t(\P^1)^s\to R$ is the natural projection. There is a proper birational morphism $\ph\colon \wt Y\to Y$ such that the map $g_2=g\circ\ph$ is regular. As $\pi \circ g_2=f_1\circ \ph$ is proper the morphism $g_2$ is proper. We have:
    \begin{align*}
        [Y,f_1^*(a), f_2\circ f_1]=[\wt Y,\ph^*f_1^*(a), f_2\circ f_1\circ \ph]=\\
        [\wt Y,g_2^*\pi^*(a), f_2\circ \pi\circ g_2]=\deg (g_2)[T\t(\P^1)^s, \pi^*(a), f_2\circ\pi].       
    \end{align*}

    Thus we have reduced the statement to the case when $Y=R\t(\P^1)^s$ and $f_1=\pi$. Let $\theta$ be some morphism $\theta\colon (\P^1)^s\to (\P^1)^s$ of degree $l, l>1$. We have
    \begin{align*}
        &[R\t (\P^1)^s, \pi^*(a), f_2\circ\pi]=\\
        &1/l[R\t (\P^1)^s,  (id\t\theta)^*\circ\pi^*(a),f_2\circ\pi\circ (id\t\theta)]=\\
        &1/l[R\t (\P^1)^s,\pi^*(a), f_2\circ\pi].
    \end{align*}

    It follows that $[R\t(\P^s),\pi^*(a), f_2\circ \pi]=0$.
\end{proof}

\begin{lemma}
\label{lemma:about_degenerate_cycles}
Let $Y$ and $R$ be varieties. Let $f_1\colon Y\to R$, $f_2\colon R\to X$ be two proper morphisms and assume that $f_1$ is dominant.  Assume that $n\leq \dim Y-\dim R$ and in the case $n =\dim Y-\dim R$ we additionally assume that $n> 0$. Let $n'\geq 0$. For any $a\in \Lambda^n(Y)$ and $b\in \Lambda(n')(R)$ we have
    $$[Y,a\wedge f_1^*(b), f_2\circ f_1]=0.$$
   \end{lemma}

\begin{remark}
    It can be deduced from this lemma that for $m< 0$ we have
    $$\Lambda(m)_j=0.$$
    The corresponding statement for Bloch's higher Chow group is clear.
\end{remark}

\begin{proof}
    We can assume that $a=\xi_1\wdw \xi_n$. Consider a rational map $g\colon Y\ard R\t(\P^1)^n$ given by $$y\mapsto (f_1(y), \xi_1(y),\dots, \xi_n(y)).$$
    
    Denote the closure of its image by $W$. There is a proper birational morphism $\ph\colon \wt Y\to Y$ such that $g\circ \ph$ is regular. We have:
    $$[Y,a\wedge f_1^*(b), f_2\circ f_1]=[\wt Y,\ph^*(a)\wedge (f_1\circ \ph)^*(b)), f_2\circ (f_1\circ \ph)].$$

    So we can assume that $g$ is regular. Denote by $\wt g$ the natural map $\wt g\colon Y\to W$ and let $i_W\colon W\to R\t(\P^1)^n$ be the canonical embedding. Denote by $\pi_{1}\colon R\t(\P^1)^n\to R$ the  projection to the first factor. Let $\pi_{2,i}\colon R\t(\P^1)^n\to \P^1$ be the projection to the $i$-th $\P^1$. Let $\wt \pi_1=\pi_1\circ i_W, \wt \pi_{2,i}=\pi_{2,i}\circ i_W$.
    
    Let $i\in \{1,\dots, n\}$. Define a rational function $\wt t_i\in\k(W)$. If $\xi_i$ is a constant, set $\wt t_i=\xi_i$. Otherwise, the map $\wt \pi_{2,i}$ is dominant and we can define $\wt t_i=\wt \pi_{2,i}^*(t)$, where $t$ is the canonical coordinate on $\P^1$. As $f_1$ is dominant, the map $\wt\pi_1$ is dominant and we have the well-defined element $\wt\pi_1^*(b)$.
    
    We have
$$a\wedge f_1^*(b)=(\wt g)^*(\wt t_1\wdw \wt t_n\wedge \wt\pi_1^*(b)).$$
    
    Consider the following two cases:

    \begin{enumerate}
        \item Let $n < \dim Y-\dim R$. We have $\dim W\leq \dim R+n<\dim Y$ and so $\dim W<\dim Y$. In this case the statement follows from the previous lemma.
        \item Let $n=\dim Y-\dim R$. If $\dim W<\dim Y$ then the proof is the same as in the previous item. Assume that $\dim W=\dim R+n=\dim Y$. This implies that $W=R\t(\P^1)^n$. We get:
        $$[Y,a\wedge f_1^*(b), f_2\circ f_1]=(\deg g)[R\t(\P^1)^n, (t_1\wdw t_n)\wedge \pi_1^*(b), f_2\circ \pi_1].$$
        So we have reduced the statement to the case when $Y=R\t(\P^1)^n, a=t_1\wdw t_n$ and $f_1=\pi_1$.
    We have
    \begin{align*}
        &[R\t(\P^1)^n, (t_1\wdw t_n)\wedge \pi_1^*(b), f_2\circ \pi_1]=\\&[R\t(\P^1)^n, (1/t_1\wdw t_n)\wedge \pi_1^*(b), f_2\circ \pi_1]=\\&-[R\t(\P^1)^n, (t_1\wdw t_n)\wedge \pi_1^*(b), f_2\circ \pi_1].
    \end{align*}
   So $[R\t(\P^1)^n, (t_1\wdw t_n)\wedge \pi_1^*(b), f_2\circ \pi_1]=0$.
\end{enumerate}
\end{proof}

We recall that we have given the definition of $\ts_D(a, f)$ in the previous section.

\begin{lemma}
\label{lemma:differential_normal_crossing_zero}
    Let $Y$ be a variety, $f\colon Y\to X$ be a proper morphism and $\xi_1,\dots, \xi_n\in \k(Y)$. Let $U$ be some open subset of $Y$. Denote by $D_i^0, D_i^\infty$ the divisors of zeros and poles of the functions $\xi_i$ on $U$. Assume that for any $1\leq n_1<n_2\dots <n_l\leq n$ and any $m_1,\dots, m_l\in \{0,\infty\}$, the codimension of each irreducible component of $D_{n_1}^{m_1}\cap\dots \cap D_{n_l}^{m_l}$ is equal to $l$. Let $\wt Y$ be a smooth variety and  $\ph\colon \wt Y\to Y$ be an alteration. Then for any divisor $E\subset \wt Y$ contracted under $\ph$ such that $\ph(E)\cap U\ne\varnothing$ we have $$\ts_E(\ph^*(\xi_1\wdw \xi_n), f\circ \ph)=0\in \L(X, m)_{j+1}.$$

\end{lemma}

\begin{proof}
When $n=0$ there is nothing to prove. So we can assume that $n\geq 1$. Denote by $W$ the intersection $\ph(E)\cap U$.
    
    We can assume that there is $k$ such that $W\subset D_i^0$ for $1\leq i\leq k$ and $W\nsubseteq D_i^m$ for any $i>k$ and any $m\in \{0,\infty\}$. 
    
    Set $a=\xi_1\wdw \xi_n, b= \xi_1\wdw \xi_k, c= \xi_{k+1}\wdw \xi_n$. 
For any $i> k$ we have $\ord_E(\ph^*(\xi_i))=0$. So we get
    $$\alpha:=\ts_{E}(\ph^*(a), f\circ \ph)=[E,\ts_{E}(\ph^*(b))\wedge \res{\ph^*(c)}{E}, \res{(f\circ \ph)}{E}].$$

Let $g_1\colon E\to W$ and $g_2\colon W\to X$ be the natural maps. We have
$$\alpha=[E,\ts_{E}(\ph^*(b))\wedge g_1^*(\res{c}{W}), g_2\circ g_1].$$
We apply Lemma \ref{lemma:about_degenerate_cycles}. Let $\dim Y=p$. We need to check that $k-1 \leq \dim E-\dim W$.

Let $T$ be some irreducible component of $\bigcap\limits_{i=1}^k D_i^0$ containing $W$. By condition of the lemma $\dim T=p-k$. So $\dim (W\cap U)\leq \dim T=p-k$. As $W$ is irreducible this implies that $\dim W=\dim W\cap U\leq p-k$. So $$\dim E-\dim W\geq p-1-(p-k)=k-1.$$
\end{proof}

We have the following lemma(see \cite[Lemma 2.14]{bolbachan_2023_chow}, however the proof given here is not correct):

\begin{lemma}
\label{lemma:char_of_strictly_regular}
    Let $Y$ be a smooth algebraic variety and $a\in\Lambda^n(Y)$ is strictly regular. Then for any closed point $y\in Y$ there is an affine open neighborhood $U$ containing $y$, such that the restriction of $a$ to $U$ can be represented as linear combination of the element of the form
    $$\xi_1\wdw \xi_k\wedge u_{k+1}\wdw u_n.$$
    Here $u_i$ are invertible on $U$ and $\xi_i$ can be extended to a local system of parameters at $y$.
\end{lemma}

\begin{proof}
    The statement follows from the following fact. Let $\alpha_1,\dots, \alpha_n$ be non-zero rational functions on $Y$ such that the divisor $\sum |\alpha_i|$ is supported on a simple crossing divisor at $y$. Then there is a regular system of parameters $\xi_1,\dots, \xi_p$ at $y$, rational functions $u_i, 1\leq i\leq n$ taking non-zero values at $y$ and integers $n_{i,j}, 1\leq i\leq n, 1\leq j\leq p$ such that for any $1\leq i\leq n$, we have: $\alpha_i=u_i\xi_1^{n_{i, 1}}\cdot\dots\cdot \xi_p^{n_{i,p}}$ .
\end{proof}

\begin{proof}[The proof of Theorem \ref{th:differential_strictly_regular_elements}]
Let $W$ be the image of $E$ under $\ph$ and $y\in W$. We apply Lemma \ref{lemma:char_of_strictly_regular} to the point $y$ and the element $a$. So we can assume that 
$$a=\xi_1\wdw \xi_k\wedge u_{k+1}\wdw u_n,$$
such that $u_i$ take non-zero values at $y$ and $\xi_i$ can be extended to local system of parameters at $y$.
Now the statement follows from Lemma \ref{lemma:differential_normal_crossing_zero}.
\end{proof}

\begin{corollary}
\label{cor:differential_exceptional_divisor_zero_Bloch}
    Let $Y$ be a variety, $f\colon Y\to X$ be a proper morphism and  $\xi_1,\dots, \xi_n\in \k(Y)$. Let $D_i^1$ be the closure of the set given by the equation $\xi_i=1$. Set $Z=\cup D_i^1$ and $U=Y\bs Z$. Denote by $D_i^0, D_i^\infty$ the divisors of zeros and poles of the functions $\xi_i$ on $U$. Assume that for any $1\leq n_1<n_2\dots <n_l\leq n$ and any $m_1,\dots, m_l\in \{0,\infty\}$, the codimension of each irreducible component of $D_{n_1}^{m_1}\cap\dots \cap D_{n_l}^{m_l}$ is equal to $l$. Let $\wt Y$ be a smooth variety and $\ph\colon \wt Y\to Y$ be an alteration. Then for any divisor $E\subset \wt Y$ contracted under $\ph$ we have $$\ts_E(\ph^*(\xi_1\wdw \xi_n), f\circ \ph)=0\in \L(X, m)_{j+1}.$$

\end{corollary}

\begin{proof}
We can assume that $n\geq 1$. Denote by $W$ the image of $E$ under $\ph$. If $W$ is not contained in $Z$ then we can apply Lemma \ref{lemma:differential_normal_crossing_zero}. Assume that $W\subset Z_i$ for some $i$. Then the rational function $\res{\ph^*(\xi_i)}{E}$ is equal to $1$ and the statement is obvious.
\end{proof}

\subsection{The differential is well-defined}
\label{sub:sec:Lambda+correctness}

The following lemma is a direct corollary of Lemma \ref{lemma:functoriality_of_residue_2} and the definition of $\L(X,m)$.

\begin{lemma}
\label{lemma:functoriality_of_residue}
    Let $Y_1$ and $Y_2$ be smooth varieties. Let $f\colon Y_2\to X$ be a proper morphism and $\ph\colon Y_1\to Y_2$ be an alteration. Let $E_1\subset Y_1, E_2\subset Y_2$ be irreducible divisors such that $\varphi(E_1)=E_2$. Then we have
    $$\ts_{E_1}(\ph^*(a),f\circ \ph)=e(E_1, E_2)f(E_1, E_2)\ts_{E_2}(a, f).$$
\end{lemma}

\begin{proposition}
\label{prop:d_is_well_defined}
    The map $d\colon \L(X,m)_j\to \L(X, m)_{j+1}$ is well-defined. 
\end{proposition}

\begin{proof}
Let $f\colon Y\to X$ be a proper morphism and $a\in\L^n(Y)$. We recall that the element $d([Y, a, f])$ is defined by the formula
$$\dfrac 1{\deg\ph}\sum\limits_{D\subset \wt Y}\ts_D(\ph^*(a), f\circ\ph),$$
where $\ph\colon \wt Y\to Y$ is any alteration such that the element $\ph^*(a)$ is strictly regular. We only need to prove that this element does not depend on $\ph$. 

Let $\ph_1\colon \wt Y_1\to Y$ and $\wt Y_2\to Y$ be two such alterations. There are alterations $\ph_{31}\colon \wt Y_3\to Y_1, \ph_{32}\colon\wt Y_{3}\to Y_2$ such that $\ph_1\circ\ph_{31}=\ph_{2}\circ\ph_{32}$ and the element $(\ph_{1}\circ\ph_{31})^*(a)=(\ph_{2}\circ\ph_{32})^*(a)$ is strictly regular. This shows that we only need to check the following statement. Let $\ph\colon \wt Y\to Y$ be an alteration and $a\in \L^n(Y)$. Assume that the elements $a$ and $\ph^*(a)$ are strictly regular. Then the following formula holds:
$$\deg(\ph)\sum\limits_{D\subset Y}\ts_D(a, f)=\sum\limits_{D\subset \wt Y}\ts_D(\ph^*(a), f\circ\ph).$$
We have:
\begin{align*}
          &\sum\limits_{D\subset \wt Y}\ts_D(\ph^*(a)), f\circ \ph)=\\
          &=\sum\limits_{D\in Div(\ph)_0} \ts_D(\ph^*(a)), f\circ \ph)+\sum\limits_{D\subset Y}\sum\limits_{D'\in Div(\ph, D)}\ts_{D'}(\ph^*(a), f\circ \ph).
\end{align*}

The first term is equal to zero by Theorem  \ref{th:differential_strictly_regular_elements}. By Lemma \ref{lemma:functoriality_of_residue} and Lemma \ref{lemma:degree_is_a_sum_multiplicties}   the second term is equal to
        \begin{align*}
            &\sum\limits_{D\subset Y}\sum\limits_{D'\in Div(\ph, D)}e(D', D)f(D',D)\ts_{D}(a, f)=\\
            &=\deg(\ph)\sum\limits_{D\subset Y}\ts_{D}(a, f).
        \end{align*}
        The proposition is proved.
\end{proof}
\subsection{The square of the differential is zero}
\label{sub:sec:Lambda_square_of_differential}
To finish the proof of Theorem \ref{th:Lambda(m)_well_defined} we need to show that $d^2=0$. This is a direct corollary of the following theorem:

\begin{theorem}
\label{th:Parshin_reciprocity_law}
    Let $Y$ be a smooth variety and $f\colon Y\to X$ be a proper morphism. Let $a\in\Lambda^n(Y)$ be strictly regular. Then
    $$\sum\limits_{D'\subset D\subset X}[D',\ts_{D'}\ts_{D}(a),f \circ j_D\circ j_{D'}]=0\in \Lambda(X, m)_{j+2}.$$
    
    Here the sum is taken over all chains $D'\subset D\subset Y$ of irreducible subvarieties of codimensions $2$ and $1$. $j_{D'}$ (resp. $j_D$) is the natural embedding of $D'$ to $D$ (resp. $D$ to $Y$). (It is clear that in this sum only finitely many non-zero terms).
\end{theorem}

The point of this theorem is that for given $D'$ there are precisely two $D$ such that $D'\subset D\subset X$ and $\ts_{D}\ts_{D'}(a)\ne 0$. Moreover these two elements have opposite signs. The particular case of this theorem when $\dim Y=2$ was proved in \cite[Theorem 2.15]{bolbachan_2023_chow}. The case $\dim Y=2, n=3$ follows from classical Parshin reciprocity law \cite{parshin1975class}.

\begin{proof}[The proof of Theorem \ref{th:Parshin_reciprocity_law}]
It is enough to show that for any subvariety $D'$  of codimension $2$ we have:
$$\sum\limits_{D'\subset D\subset X}[D',\ts_{D'}\ts_{D}(a),f \circ j_D\circ j_{D'}]=0\in \Lambda(X, m)_{j+2}.$$
(The sum is taken over all irreducible divisors $D$ containing $D'$).

    Choose some point $y\in D'$. We can apply Lemma \ref{lemma:char_of_strictly_regular} and assume that $a=\xi_1\wdw \xi_k\wedge u_{k+1}\wdw u_n$ at some open affine neighborhood $U$ of the point $y$. Moreover, we can assume that the divisors $D_i=(\xi_i)\cap U$ are irreducible. The only case when $\ts_{D'}\ts_{D}(a)\ne 0$ is when $D'\cap U=D_{i_1}\cap D_{i_2}$ and $D\cap U=D_{i_1}$ for some $i_1\ne i_2$. Without loss of generality we can assume that $i_1=1, i_2=2$. 
    
    It remains to show that $\ts_{D_{1}\cap D_{2}}\ts_{D_{1}}(a)=-\ts_{D_{1}\cap D_{2}}\ts_{D_{2}}(a)$. Denote $H=D_1\cap D_2$. We have:
    \begin{align*}
        &\ts_{H}\ts_{D_1}(a)=\\
        &\ts_{H}(\res{\xi_2}{D_1}\wdw \res{\xi_k}{D_1}\wedge \res{u_{k+1}}{D_1}\wdw \res{u_{n}}{D_1})=\\
        &\res{\xi_3}{H}\wdw \res{\xi_k}{H}\wedge \res{u_{k+1}}{H}\wdw \res{u_{n}}{H}.
    \end{align*}
    The last formula holds because $\ord_{H}(\res{\xi_2}{D_1})=1$. In the same way we get 
\begin{align*}
        &\ts_{H}\ts_{D_2}(a)=\\
        &-(\res{\xi_3}{H}\wdw \res{\xi_k}{H}\wedge \res{u_{k+1}}{H}\wdw \res{u_{n}}{H}).
    \end{align*}
 \end{proof}
The statement follows.

\begin{proof}[The proof of Theorem \ref{th:Lambda(m)_well_defined}]
    By Proposition \ref{prop:d_is_well_defined} we already know that $d$ is well-defined. Let us prove that $d^2=0$. We need to show that $d^2([Y,a,f])=0$. We can assume that $a$ is strictly regular. To apply Theorem \ref{th:Parshin_reciprocity_law} we only need to prove that for any $D\subset Y$ we have
    $$d(\ts_D(a, f))=\sum\limits_{D'\subset D\subset Y}[D',\ts_{D'}\ts_{D}(a),f \circ j_D\circ j_{D'}].$$
    To prove this formula it is enough to check that  $\ts_{D}(a)$ is strictly regular. This follows from the following fact: if $D_1+\dots+D_n$ is a simple normal crossing divisor, then $D_1\cap D_l+\dots+D_{l-1}\cap D_l$ is a simple normal crossing divisor on $D_l$.
\end{proof}

\section{Properties of the complex $\Lambda(X, m)$}
\label{sec:Lambda_properties}

\subsection{Proper pushforward}
\label{sub:sec:proper_push_forward}
Let $X_1, X_2$ be algebraic schemes and $\psi\colon X_1\to X_2$ be a proper morphism. Denote $k=\dim X_2-\dim X_1$. Define a map

$$\psi_*\colon \Lambda(X_1, m-k)[-2k]\to \Lambda(X_2, m)$$

 by the formula
$$\psi_*([Y, a, f])=[Y, a, \psi\circ f].$$

The fact this map is a well-defined morphism of complexes is clear.

\subsection{Flat pullback}
\label{sub:sec:flat_pullback}
Let $\psi\colon X_1\to X_2$ be a flat morphism of relative dimension $k$. Define $$\psi^*\colon \Lambda(X_2, m)\to \Lambda(X_1, m)$$ 
as follows.
Let $[Y_2, a, f]\in \L(X_2, m)$.

Consider the pullback diagram

\begin{equation*}
    \begin{tikzcd}
        Y_1 \ar[r, "\psi_f"]\ar[d, "f_\psi"] & Y_2\ar[d, "f"]\\
        X_1\ar[r, "\psi"] & X_2
    \end{tikzcd}
\end{equation*}

In this diagram $\psi_f, f_\psi$ are the natural projections. Let $p=\dim Y_2$. Let $$[Y_1]_{p+k}=\sum\limits_{\alpha\in A}n_\alpha [Y_1^\alpha]_{p+k}.$$ (See \cite{fulton2013intersection}). In this formula $[Z]_l$ is $l$-dimensional cycle associated to a closed subscheme of dimension $\leq l$. The sum is taken over all irreducible component of $X_1\t_{X_2} Y_2$ and $n_\alpha$ is multiplicity of $Y_1$ along $Y_1^\alpha$. Denote by $\psi_\alpha$ (resp. $f_\alpha$) the restriction of the map $\psi_f$ (resp. $f_\psi$) to $Y_1^\alpha$. 

 As $f_\psi$ is proper, each of the maps $f_\alpha$ is proper. As the map $\psi_f$ is flat, each of the maps $\psi_\alpha$ is dominant. Define:

$$\psi^*([Y_2, a, f])=\sum\limits_{\alpha\in A} n_\alpha[Y_1^\alpha, \psi_\alpha^*(a), f_{\alpha}]\in \L(X_1, m).$$

\begin{proposition}
\label{prop:pull_back_well_defined}
    The map $\psi^*$ is well-defined.
\end{proposition}

We need the following lemma.

\begin{lemma}
\label{lemma:map_on_each_component_dominant}
    Let $Y_2$ and $\wt Y_2$ be varieties and $Y_1$ be an algebraic scheme. Let $g\colon \wt Y_2\to Y_2$ be a dominant proper morphism and $h\colon Y_1\to Y_2$ be a flat morphism. Denote $Y_1\t_{Y_2}\wt Y_2$ by $\wt Y_1$. The image of every irreducible component of $\wt Y_1$ under the map $g_h$ is an irreducible component of $Y_1$.
    \begin{equation*}
    \begin{tikzcd}
        \wt Y_1\arrow[d, "g_h"]\arrow[r, "h_g"]&\wt Y_2\arrow[d, "g"]\\
        Y_1\arrow[r, "h"] & Y_2
    \end{tikzcd}
\end{equation*}
\end{lemma}

\begin{proof}
    Let $Z\subset \wt Y_1$ be some irreducible component. As $g$ is proper, $g_h$ is proper as well. Hence $g_h(Z)$ is closed. It remains to show that $g_h(Z)$ contains some open subset of $Y_1$.
    
    As $h_g$ is flat, it is open and so $h_g(Z)\subset \wt Y_2$ is dense. As $g$ is dominant, this shows that the image of $Z$ in $Y_2$ is dense. This shows that for any open subset $U\subset Y_2$ the base change $Z_U=Z\times_{Y_2} U$ is a non-empty open subset of $Z$. By general flatness we can choose $U$ in such a way that the base change of $g$ to $U$ would be flat. So we can assume $g$ to be flat. In this case $g_h$ is also flat and hence open. Hence $g_h(Z)$ contains some open subset of $Y_1$. 
\end{proof}

\begin{proof}[The proof of Proposition \ref{prop:pull_back_well_defined}]Let $\ph\colon \wt Y_2\to Y_2$ be an alteration. Denote the composition $f\circ\ph$ by $\wt f$. We get the following diagram:

\begin{equation*}
    \begin{tikzcd}
        \wt Y_1 \ar[r, "\psi_{\wt f}"]\ar[d, "\ph_\psi"] & \wt Y_2\ar[d, "\ph"]\\
        Y_1 \ar[r, "\psi_f"]\ar[d, "f_\psi"] & Y_2\ar[d, "f"]\\
        X_1\ar[r, "\psi"] & X_2
    \end{tikzcd}
\end{equation*}

In this diagram $\wt Y_1$ is the fiber product $Y_1\times _{Y_2}\wt Y_2$ and the maps $\psi_{\wt f}, \ph_\psi$ are the natural projections. Let $$\psi_{\wt f}^{*}([\wt Y_2]_p)=\sum\limits_{\beta\in B}m_\beta [\wt Y_1^\beta]_{p+k}.$$ Denote the restriction of $\psi_{\wt f}$ to $\wt Y_1^\beta$ by $\psi_{\beta}$. Denote by $\wt f_\beta$ the restriction of the map $ f_\psi\circ \ph_\psi$ to $
\wt Y_1^\beta$.

It follows from Lemma \ref{lemma:map_on_each_component_dominant}, that for any $\beta\in B$, there is a unique $\alpha\in A$ such that $\ph_\psi(\wt Y_1^{\beta})= Y_1^{\alpha}$. Denote this index by $\theta(\beta)\in A$. Denote by $\ph_\beta$ the corresponding map $\wt Y_1^\beta\to Y_1^{\theta(\beta)}$.

Let us apply push and pull formula to the fibered diagram

\begin{equation}
   \begin{tikzcd}
	\wt Y_1 & \wt Y_2\\
 Y_1 & Y_2
	\arrow["\psi_{\wt f}", from=1-1, to=1-2]
	\arrow["\varphi", from=1-2, to=2-2]
	\arrow["\varphi_\psi", from=1-1, to=2-1]
	\arrow["\psi_f", from=2-1, to=2-2]
\end{tikzcd} 
\end{equation}

We get 
\begin{align*}
 &\varphi_*([\wt Y_2]_p)=(\deg \ph)[Y_2]_p,\quad
 (\ph_\beta)_*([\wt Y_1^\beta]_{p+k})=\deg(\ph_\beta)[Y_1^{\theta(\beta)}]_{p+k},\\
 &(\psi_{\wt f})^*([\wt Y_2]_p)=\sum\limits_{\beta\in B} m_\beta[\wt Y_1^\beta]_{p+k},\quad
 \psi_f^*([Y_2]_p)=\sum\limits_{\alpha \in A} n_\alpha[Y_1^\alpha]_{p+k}. 
\end{align*}
 So push and pull formula $(\psi_{f})^*\circ \varphi_* = (\varphi_{\psi})_*\circ(\psi_{\wt f})^*$ implies

\begin{align*}
    &(\deg \ph)\sum\limits_{\alpha\in A} n_\alpha[Y_1^\alpha]_{p+k}= \sum\limits_{\beta\in B} m_\beta(\varphi_{\psi})_*([\wt Y_1^\beta]_{p+k})=\\
    &\sum\limits_{\alpha\in A} \sum\limits_{\beta\in\theta^{-1}(\alpha)} m_\beta\deg(\ph_\beta)[Y_1^\alpha]_{p+k}.
\end{align*}

So for any $\alpha$

$$n_\alpha\deg\ph=\sum\limits_{\beta\in\theta^{-1}(\alpha)} m_\beta\deg(\ph_\beta).$$

Using this formula we get
\begin{align*}
    &\psi^*([\wt Y_2, \ph^*(a), \wt f])=\sum\limits_{\beta\in B}m_\beta[\wt Y_1^\beta, (\psi_\beta^*\circ \ph^*)(a), \wt f_\beta]=\\
    &\sum\limits_{\beta\in B} m_\beta \deg \ph_\beta[ Y_1^{\theta(\beta)}, \psi_{\theta(\beta)}^*(a), f_{\theta(\beta)}]=(\deg\phi)\sum\limits_{\alpha\in A}n_\alpha[Y_1^\alpha, \psi_\alpha^*(a), f_\alpha].
\end{align*}

So the map $\psi^*$ is well-defined. 
\end{proof}

Let $Y$ be a variety. For a Cartier divisor $D$ on $Y$, denote the corresponding Weil divisor by $[D]$. 
We need the following lemma \cite[Proposition 1.4]{fulton2013intersection}:

\begin{lemma}
\label{lemma:degree_of_morphism}
    Let $Y_1$, $Y_2$ be varieties and $\varphi\colon Y_1\to Y_2$ be an alteration. Then for any Cartier divisor $D$ on $Y_2$ we have
    $$\varphi_*[\varphi^*(D)]=(\deg\varphi)[D].$$ 
    In this formula $\ph^*$ is pullback on Cartier divisors and $\ph_*$ is pushforward on Weil divisors.
\end{lemma}

\begin{proposition}
    The map $\psi^*$ is a morphism of complexes.
\end{proposition}

\begin{proof}
Let $x=[Y_2,a,f]$. We can assume that this element is strictly regular. We have

$$d(x)=d([Y_2,a,f])=\sum\limits_{D\subset Y_2}[D, \ts_D(a)].$$

The base change $D\t_{X_2}X_1$ is equal to scheme-theoretic pre-image $\psi_f^{-1}(D)$. According to Lemma 1.7.2 from \cite{fulton2013intersection} the cycle $[(\psi_f)^{-1}(D)]_{p+k-1}$ is equal to 
$$\sum\limits_{\alpha\in A}n_{\alpha} (i_\alpha)_*([\psi_\alpha^{-1}(D)]_{p+k-1}).$$
In this formula $i_\alpha$ is the natural embedding $Y_1^\alpha\emb Y_1$.  For a divisor $D'\in\Div(\psi_\alpha, D)$, denote by $w_{D'}$ the natural map $D'\to D$. The previous discussion shows that for any element $b\in \L^n(D)$
$$\psi^*([D, b])=\sum\limits_{\alpha\in A}\sum\limits_{D'\in \Div(\psi_\alpha, D)}n_\alpha\mult(D', \psi_\alpha^*(D))[D', w_{D'}^*(b)].$$

Applying this statement to $b=\ts_D(a)$, we get

$$\psi^*(d([Y,a,f]))=\sum\limits_{D\subset Y_2}\sum\limits_{\alpha\in A}\sum\limits_{D'\in \Div(\psi_\alpha, D)}n_\alpha\mult(D', \psi_\alpha^*(D))[D', w_{D'}^*(\ts_D(a))].$$

On the other hand

$$\psi^*([Y,a, f])=\sum\limits_{\alpha\in A}n_\alpha [Y_1^\alpha, \psi_\alpha^*(a)].$$

Let $\ph_\alpha\colon \wt Y_1^\alpha\to Y_1^\alpha$ be an alteration such that the element $\ph_\alpha^*(\psi_\alpha^*(a))$ is strictly regular. Denote $g_\alpha=\psi_\alpha\circ\ph_\alpha$ We get

$$d(\psi^*(x))=\sum\limits_{\alpha\in A}n_\alpha\dfrac 1{\deg\ph_\alpha}\sum\limits_{D\subset \wt Y_1^\alpha}[D, \ts_C((g_\alpha)^*(a))].$$

So it is enough to check that for any $\alpha\in A$

$$\sum\limits_{D\subset Y_2}\sum\limits_{D'\in \Div(\psi_\alpha, D)}\mult(D', \psi_\alpha^*(D))[D', w_{D'}^*(\ts_{D}(a)))] = \dfrac 1{\deg\ph_\alpha}\sum\limits_{D''\subset \wt Y_1^\alpha}[D'', \ts_{D''}(g_\alpha^*(a))]$$

We have

$$\sum\limits_{D''\subset \wt Y_1^\alpha}[D'', \ts_{D''}(g_\alpha^*(a))]=\sum\limits_{D''\in Div(\ph_\alpha)_0}[D'', \ts_{D''}(g_\alpha^*(a))]+\sum\limits_{D''\in Div(\ph_\alpha)}[D'', \ts_{D''}(g_\alpha^*(a))].$$

The first sum is equal to zero by Lemma \ref{lemma:differential_normal_crossing_zero} and by the fact that $\psi$ is flat of relative dimension $k$. Let $D''\in Div(\ph_\alpha)$. The codimension of $g_\alpha(D'')$ cannot be bigger then $1$ as $\psi_f$ is flat. If this codimension is equal to zero, then $[D'', \ts_{D''}(g_\alpha^*(a))]=0$. So we get

$$\sum\limits_{D''\in Div(\ph_\alpha)}[D'', \ts_{D''}(g_\alpha^*(a))]=\sum\limits_{D\subset Y_2}\sum\limits_{D''\in \Div(g_\alpha, D)}[D'', \ts_{D''}(g_\alpha^*(a))].$$

So it remains to check that for any $D\subset Y_2$ and any $\alpha \in A$ we have

$$\sum\limits_{D'\in \Div(\psi_\alpha, D)}\mult(D', \psi_\alpha^*(D))[D', w_{D'}^*(\ts_{D}(a))] =\dfrac 1{\deg g_\alpha}\sum\limits_{D''\in \Div(g_\alpha, D)}[D'', \ts_{D''}(g_\alpha^*(a))].$$

For a divisor $D''\in \Div(g_\alpha, D)$ denote the natural map $D''\to D$ by $h_{D''}$. By Lemma \ref{lemma:functoriality_of_residue_2}, we get

\begin{align*}
   &\sum\limits_{D''\in \Div(g_\alpha, D)}[D'', \ts_{D''}(g_\alpha^*(a))]=\sum\limits_{D'\in \Div(\psi_\alpha, D)}\sum\limits_{D''\in \Div(\ph_\alpha, D')}[D'', \ts_{D''}(g_\alpha^*(a))] =\\
   & \sum\limits_{D'\in \Div(\psi_\alpha, D)}\sum\limits_{D''\in \Div(\ph_\alpha, D')}\mult(D'', g_\alpha^{*}(D))[D'', h_{D''}^*(\ts_D(a))] =\\
   &\sum\limits_{D'\in \Div(\psi_\alpha, D)}\sum\limits_{D''\in \Div(\ph_\alpha, D')}\mult(D'', g_\alpha^{*}(D))f(D'', D')[D', w_{D'}^*(\ts_D(a))].
\end{align*}

By Lemma \ref{lemma:degree_of_morphism}, we know that $(\ph_\alpha)_*[\ph_\alpha^*(\psi_\alpha^*(D))]=(\deg\ph_\alpha)[\psi_\alpha^*(D)]$. This implies that

$$\sum\limits_{D''\in \Div(\ph_\alpha, D')}\mult(D'', g_\alpha^{*}(D))\deg{D''/D'}=(\deg \ph_\alpha)\mult(D', \psi_\alpha^*(D)).$$

The statement follows.   
\end{proof}

\begin{proposition}
    If $\psi_1, \psi_2$ are flat morphism, then $\psi_2^*\psi_1^*=(\psi_1\psi_2)^*$.
\end{proposition}

\begin{proof}
    Follows from functoriality of flat pullback on algebraic cycles (see \cite{fulton2013intersection}).
\end{proof}

\begin{corollary}
\label{cor:Galois_descent}
We have:
\begin{enumerate}
    \item  Let $X$ be a variety over $L_1$ and $L_1\subset L_2$ be a finite extension. Denote the natural map $X\t_{\spec L_1} \spec L_2\to X$ by $\psi$. Then $\psi_*\psi^*$ is a multiplication by $\deg L_2/L_1$. 
    \item If $L_2/L_1$ is Galois then $\L^{*}(\spec L_2, m)^{Gal(L_2/L_1)}=\L^{*}(\spec L_1, m)$.
\end{enumerate}
\end{corollary}

\begin{proof}
    \begin{enumerate}
        \item We have
        \begin{align*}
            &\psi_*\psi^*([Y, a, f])=\sum\limits_{\alpha\in A}n_\alpha[Y_1^\alpha, \psi_\alpha^*(a), \psi\circ f_\alpha]=\\
            &\sum\limits_{\alpha\in A}n_\alpha[Y_1^\alpha, \psi_\alpha^*(a), f\circ \psi_\alpha]=[Y, a, f]\sum\limits_{\alpha\in A}n_\alpha\deg(\psi_\alpha)=[Y,a,f]\deg(L_2/L_1).
        \end{align*}
    \item We recall that $\L(X, m)$ is a vector space over $\Q$. It follows from the previous item that the natural map  $$\L^*(\spec L_1, m) \to \L^*(\spec L_2, m)^{Gal(L_2/L_1)}$$
    is injective.
    It remains to show that
    $$\psi^*\psi_*=\sum\limits_{g\in Gal(L_2/L_1)}g^*.$$
    The proof of this formula is standard.
    \end{enumerate}
\end{proof}

\subsection{Localization}
\label{sub:sec:localization}
\begin{theorem}
\label{th:localization}
    Let $i\colon Z\to X$ be a closed embedding of codimension $k$ closed subset (not necessarily irreducible). Denote by $j$ the corresponding open embedding $U\to X$. The following sequence is exact:
    $$0\to \L(Z, m-k)[-2k]\xrightarrow{i_*}\L(X, m)\xrightarrow{j^*} \L(U,m)\to 0. $$
\end{theorem}

\begin{remark}
    The corresponding statement for Bloch's higher Chow group was proved in \cite{bloch_1994_moving_lemmas}. It is surprising that in our case the sequence is exact on level of complexes.
\end{remark}

\begin{lemma}
    The map $i_*$ is injective (as a map of complexes).
\end{lemma}

\begin{proof}
    Define a map $\theta\colon \L(X, m)\to \L(Z, m-k)[-2k]$ as follows. Let $\xi =[Y, a, f]\in \L(X, m)$. If the closure of $f(Y)$ is not contained in $Z$ then $\theta(\xi)=0$. Otherwise $f$ can be factored as $f=i\circ f'$ and define $\theta(\xi)= [Y, a, f']$. It is easy to see that this map is well-defined. We have $\theta\circ i_*=id$ and so $i_*$ is injective.
\end{proof}

\begin{proof}[The proof of Theorem \ref{th:localization}]
\begin{enumerate}
\item It follows from the previous lemma that we need to show that the following map is an isomorphism:
$$\tau\colon \L(X,m)/(\im i_*)\to \L(U,m).$$

\item Define an inverse map $\varphi\colon \L(U,m)\to \L(X,m)/(Im(i_*)))$. Let $[Y, a, f]\in\L(U,m)$. By Nagata's compactification theorem \cite{nagata1963generalization} the map $j\circ f$ can be factored as $\wt f\circ \wt j$, where $\wt j\colon Y\to \wt Y$ is an open embedding and $\wt f\colon \wt Y\to  X$ is proper. 
\begin{equation*}
    \begin{tikzcd}
        Y \arrow[r, hook, "\wt j"]\ar[d, "f"]& \wt Y\ar[d, "\wt f"]\\
        U \arrow[r, hook, "j"]& X
    \end{tikzcd}
\end{equation*}

Define $\varphi([Y, a, f])=[\wt Y, (\wt j^*)^{-1}(a), \wt f]$. 

    \item We need to show that $\varphi$ does not depend on the choice of a compactification. Let $(\wt Y_1, \wt j_1, \wt f_1), (\wt Y_2, \wt j_2, \wt f_2)$ be two compactifications such that there is a morphism over $X, \phi\colon \wt Y_2\to \wt Y_1$ such that $\phi\circ \wt j_2=\wt j_1$ and $\wt f_2=\wt f_1\circ\phi$. Let us show that $$[\wt Y_1, (\wt j_1^*)^{-1}(a), \wt f_1]=[\wt Y_2, (\wt j_2^*)^{-1}(a), \wt f_2].$$
Indeed, we have
    $$[\wt Y_1, (\wt j_1^*)^{-1}(a), \wt f_1]=[\wt Y_2, \phi^*((\wt j_1^*)^{-1}(a)), \wt f_1\circ\phi]=[\wt Y_2, (\wt j_2^*)^{-1}(a), \wt f_2].$$
The check that the relations are satisfied is an easy exercise.

    \item Let us prove that $\varphi\circ\tau=id$. Let $c=[Y,a, f]\in \L(X, m)$. Let $Y_U=Y\times_X U$, and $f_j\colon Y_U\to U, j_f\colon Y_U\to Y$ be the corresponding projections. 
    \begin{equation*}
        \begin{tikzcd}
            Y_U\ar[r, "j_f"]\ar[d, "f_j"] & Y\ar[d, "f"]\\
            U\ar[r,"j"] & X
        \end{tikzcd}
    \end{equation*}
    
    We have $\tau(c)=[Y_U, j_f^*(a), f_j]$. We have  $j\circ f_j = f\circ j_f$. So in this case we can take $\wt{Y}=Y, \wt j=j_f, \wt f=f$. So we get $\ph(\tau(c))=\ph([Y_U, j_f^*(a), f_j])=[Y, (j_f^*)^{-1}(j_f^*(a)), f])=[Y, a, f]$.

    \item It remains to show that $\tau\circ\varphi=id$. Let $$c=[Y,a, f]\in \L(U,m).$$ We have $$\ph(c)=[\wt Y, (\wt j^*)^{-1}(a),\wt f].$$ Let $\wt Y_U=\wt Y\times_{X}U$. Let $r\colon Y\to \wt Y_U, \wt f_U\colon \wt Y_U\to U$ be the natural maps. We get $$\tau(\ph(q))=[\wt Y_U, (r^*)^{-1}(a), \wt f_U].$$
    
    It is not difficult to show that $r$ is an isomorphism. So we get:
$$[\wt Y_U, (r^*)^{-1}(a), \wt f_U]=[Y,a,f].$$
    
\end{enumerate}

\end{proof}

\subsection{The complex $\Lambda(X,m)$ and Bloch's higher Chow group.}
We recall that we have given the definition of cubical higher Chow group in Section \ref{sub:sec:preliminary_def}.

Define a morphism of complexes $\mc W(m)\colon CH(m)\to\Lambda (m)$ as follows. Let $Z\in  z_p(X, n)$ be an irreducible cycle. Denote by $\wt Z$ the closure of this cycle in $X\t(\P^1)^n$. Define
$$\mc W(m)(Z)=[\wt Z,  \wt t_1\wedge\dots\wedge \wt t_n, f_Z].$$
In this formula $f_Z$ it the natural projection $\wt Z\to X$ and $\wt t_i$ are the restrictions of the standard coordinates on $(\P^1)^n$ to $\wt Z$. Here is the main result of this subsection:

\begin{theorem}
\label{th:W(m)_is_morphism}
    The map $\mc W(m)$ is a morphism of complexes.
\end{theorem}

\begin{proof}
    Denote by $\pi$ the natural projection $\wt Z \to X$. Denote by $\wt t_i\in\k(\wt Z)$ the restriction of the coordinate function $t_i$ to $\wt Z$. Let $a=\wt t_1\wdw \wt t_n$ and $a_i = \wt t_1\wdw \wt t_{i-1}\wedge \wt t_{i+1}\wdw \wt t_n$. Let  $\ph\colon T\to \wt Z$ be an alteration with smooth $T$ such that the element $\wt a:=\ph^*(a)$ is strictly regular. Let $f=\pi\circ \ph$. We have
    $$d(\mc W(Z))=\dfrac 1{\deg\ph}\sum\limits_{D\subset T}\ts_D(\wt a, f)=\dfrac 1{\deg\ph}\sum\limits_{D\in \Div(\ph)_0}\ts_D(\wt a, f)+\dfrac 1{\deg\ph}\sum\limits_{D\in \Div(\ph)}\ts_D(\wt a, f).$$

    In this formula $Div(\ph)_0$ is the set of divisors contracted under $\ph$ and $Div(\ph)$ is the set of  divisors which are not contracted under $\ph$. By Corollary \ref{cor:differential_exceptional_divisor_zero_Bloch}, the first item is equal to zero. For any divisor $D\subset \wt Z$, the set of divisors $D'\in T$ such that $\ph(D')=D$ is denoted by $Div(\ph, D)$. We get
    $$\dfrac 1{\deg\ph}\sum\limits_{D\in \Div(\ph)}\ts_D(\wt a, f)=\dfrac 1{\deg\ph}\sum\limits_{D\subset \wt Z}\sum\limits_{D'\in \Div(\ph, D)}\ts_{D'}(\wt a, f).$$

    On the other hand:
    $$\mc W(d(Z))=\sum\limits_{i=1}^n\sum\limits_{\substack{D\subset \wt Z\\\ord_D(\wt t_i)\ne 0}}(-1)^{i+1}\ord_{D}(\wt t_i)[D,i_D^*(a_i), \pi \circ  i_D].$$
    In this formula $i_D$ is the embedding of $D$ into $\wt Z$. As $Z$ intersects all the faces properly, for any divisor $D\subset \wt Z$ there is at most one $i$, such that $\ord_{D}(\wt t_i)\ne 0$. Consider the following two cases:
    \begin{enumerate}
        \item For any $i$ we have $\ord_{D}(\wt t_i)=0$. In this case for any $D'\in\Div(\ph, D)$ we have $\ts_{D'}(\wt a, f)=0$.
        \item There is some $i$ such that $\ord_{D}(\wt t_i)\ne 0$. It remains to check the following formula:
            $$(-1)^{i+1}\ord_{D}(\wt t_i)[D,i_D^*(a_i), \pi\circ i_D]=\dfrac 1{\deg\ph}\sum\limits_{D'\in Div(\ph, D)}\ts_{D'}(\wt a, f).$$
    Denote by $g_{D'}$ the natural map $D'\to \wt Z$.
    As $\ord_{D'}(\ph^*(\wt t_j))=0$ for any $j\ne i$, it follows from the definition of tame symbol that
    \begin{align*}
        \ts_{D'}(\wt a, f)=(-1)^{i+1}\ord_{D'}(\ph^*(\wt t_i))[D', g_{D'}^*(a_i), \pi\circ g_{D'}]=\\
        (-1)^{i+1}\ord_{D'}(\ph^*(\wt t_i))f(D',D)[D,i_D^*(a_i), \pi\circ i_D].
    \end{align*}
    It remains to show that
    $$\ord_{D}(\wt t_i)=\dfrac 1{\deg\ph}\sum\limits_{D'\in Div(\ph, D)}f(D',D)\ord_{D'}(\ph^*(\wt t_i)).$$
    This follows from Lemma \ref{lemma:degree_of_morphism}.
    \end{enumerate}
\end{proof}

\section{The generators in the case $X=\spec \k$}
\begin{theorem}
\label{th:Lambda_is_generated_by_rationally_connected}
    Let $\k$ be any field. The group $\L(\k, m)/\im(d)$ is generated by elements of the following form
    $$[X\t\P^1, w], w\in \L^n(X\t\P^1).$$
\end{theorem}

The proof of this theorem is the only place where we are using the Hironaka theorem on resolution of singularities \cite{hironaka_1}.

\begin{proof}
    Denote by $A$ a subgroup of $\L(\k, m)$ generated by elements of the form $[X\t\P^1, w]$. Define an increasing filtration $\mc G_*$ on $\L(\k,m)_j/\im(d)$ as follows. The vector space $\mc G_s(\L(\k,m)_j/\im(d))$ is generated by the elements $[Y,a]$ such that there is a rational map $Y\to(\P^1)^p$ of degree $\leq s$. We will prove by induction on $s$ that $\mc G_s$ coincides with $A$. The base $s=1$ is clear. 
    
    Let us prove the inductive step. We can assume that there is a rational map $f\colon Y\to (\P^1)^p$ of degree $s$ . This means that the variety $Y$ is birational to a hypersurface $Y'$ in $(\mb P^1)^{p+1}$ given by an equation of the form
    $$P(x_1,\dots, x_{p+1}): = x_1^s+\sum\limits_{i=0}^{s-1}x_1^iP_i(x_2,\dots, x_{p+1})=0.$$
    In this formula $P_i$ are some elements in $\k(x_2,\dots, x_{p+1})$. We need to show that for any $\alpha_1,\dots, \alpha_n\in \k(Y')$, the element $a:=[Y', \alpha_1\wdw \alpha_n]$ lies in $A$.

    There are rational functions $Q_j, 1\leq j\leq n$ of the form 
    
    $$Q_j(x_1,\dots, x_{p+1})=\sum\limits_{i=0}^{s-1}x_1^iQ_{j,i}(x_2,\dots x_{p+1}),$$
    such that the restrictions of $Q_j$ to $Y'$ coincide with $\alpha_j$. Here $Q_{j,i}\in \k(x_2,\dots, x_{p+1})$.
    
    Consider the element
    $$b=[(\mb P^1)^{p+1}, P\wedge Q_1\wdw Q_n].$$ Denote the element $P\wedge Q_1\wdw Q_n$ by $\lambda$. Choose some proper birational morphism $\ph\colon S\to (\mb P^1)^{p+1}$ given as composition of blow-ups in smooth centers such that the element $\ph^*(\lambda)$ is strictly regular. Let us compute $d([S, \ph^*(\lambda)])$. 

    We have
    \begin{align*}
        &d([S, \ph^*(\lambda)])=\\
        &=a+\sum\limits_{\substack{D\subset (\mb P^1)^{p+1}\\D\ne Y'}}[D,\ts_D(\lambda)]+\sum\limits_{\substack{D\subset S\\\dim(\ph(D))<\dim D}}[D,\ts_D(\ph^*(\lambda))].
    \end{align*}
    
    By the inductive assumption all the terms from the first sum lie in $A$. The terms from the second sum lie in $A$ since the exceptional divisor of any blow-up along a smooth center is always birational to $S'\t \P^1$ for some $S'$.
\end{proof}

\section{The complex $\Lambda(X,m)$ and polylogarithmic complex.}
The definition of the complex $\Gamma(F, m)$ can be found in \cite{goncharov1994polylogarithms}. This complex looks as follows:
$$\Gamma(F,m)\colon \mathcal B_m(F)\xrightarrow{\delta_m} \mathcal B_{m-1}(F)\otimes F^\times\xrightarrow{\delta_m}\dots\xrightarrow{\delta_m}\mathcal B_2(F)\otimes \Lambda^{m-2}F^\times\xrightarrow{\delta_m}\Lambda^m F^\times.$$

This complex is concentrated in degrees $[1,m]$. The group $\mathcal B_m(F)$ is the quotient of the free abelian group generated by symbols $\{x\}_m, x\in \mathbb P^1(F)$ by some explicitly defined subgroup $\mathcal R_m(F)$ (see \cite{goncharov1994polylogarithms}).   The differential is defined as follows: $\delta_m(\{x\}_k\otimes x_{k+1}\wedge \dots \wedge x_m)=\{x\}_{k-1}\otimes x\wedge x_{k+1}\wedge \dots \wedge x_m$ for $k>2$ and $\delta_m(\{x\}_2\otimes x_3\wdw x_m)=x\wedge (1-x)\wedge x_3\wdw x_m$.

Everywhere in this paper we can replace the complex $\Gamma(F,m)$ with its canonical truncation $\tau_{\geq m-1}\Gamma(F,m)$. Therefore, only the definition of the group $\mathcal R_2(F)$ is relevant for us. As it was noted in Section 4.2 of \cite{goncharov1994polylogarithms} this group  is generated by the following elements:

$$\sum\limits_{i=1}^5(-1)^i\{c.r.(x_1,\dots, \widehat x_i,\dots, x_5)\}_2, \{0\}_2, \{1\}_2, \{\infty\}_2.$$
  In this formula $x_i$ are five different points on $\mathbb P^1$ and $c.r.(\cdot)$ is the cross ratio.

Let $\k$ be an arbitrary field. We recall that in the case $X=\spec \k$, we denote the element $[Y, a, f]\in \L(X, m)$ simply by $[Y,a]$.  Define a morphism of complexes $$\mc T_{\ge m-1}(m)\colon \tau_{\ge m-1}\Gamma(\k,m)\to \tau_{\ge m-1}\Lambda(\k, m)$$ as follows. The element $\{a\}_2\wedge c_3\wdw c_m$ goes to
$$[\P^1, t\wedge (1-t)\wedge (1-a/t)\wedge c_3\wdw c_m].$$
The element $c_1\wdw c_m$ goes to
$[\spec \k, c_1\wdw c_m]$. We call the map $\mc T_{\geq m-1}(m)$ by \emph{Totaro map}.

\begin{theorem}
    \label{th:Tatoro_is_morphism_of_complexes}
    Let $F$ be an arbitrary field of characteristic zero. The map $\mc T_{\ge m-1}(m)$  is a quasi-isomorphism. In particular for $j=m-1, m$ we have $$H^j(\Lambda(\spec F, m))\cong H^j(\Gamma(F, m)).$$
     Moreover, when $\k$ is algebraically closed the map $\mc T_{\ge m-1}(m)$ is an isomorphism of complexes.
\end{theorem}

The rest of this section is devoted to the proof of this theorem.

\subsection{The map $\mc T_{\ge m-1}$ is well-defined}

Let $V$ be $2$-dimensional vector space over $\k$ and $l_i\in V^{*}$. Denote by $T(l_1, l_2, l_3, l_4)$ the element $[\P(V), \omega(l_1, \dots, l_4)]$, where $$\omega(l_1, l_2, l_3, l_4)=\dfrac {l_1}{l_4}\wedge \dfrac {l_2}{l_4}\wedge \dfrac {l_3}{l_4}.$$

Let $a\in \k\bs\{0\}$. Define $T_a^{2}$
$$T_a^{2} = [\P^1, (t\wedge (1-t)\wedge (1-a/t))].$$

\begin{lemma}
\label{lemma:abel_five_term_relation_goes_to_zero}
    The following statements true:
    \begin{enumerate}
        \item We have: $$T_a^2=-T_{1/a}^2=-T_{1-a}^2.$$
        \item We have
        $$T(l_1, l_2, l_3, l_4)=T_{c.r.(\pi(l_1),\pi(l_2), \pi(l_3), \pi(l_4))}^2.$$
        In this formula $\pi\colon V^*\bs \{0\}\to \P(V^*)$ is the natural projection and $c.r.(\cdot)$ is the cross-ratio.
         \item Let $l_1,\dots, l_5\in V^*$. Assume that any two of these vectors are linearly independent. Then we have
        $$\sum\limits_{i=1}^5(-1)^iT(l_1,\dots\hat{l_i},\dots l_5)=0.$$
    \end{enumerate}
\end{lemma}

\begin{corollary}
\label{cor:Abel_five_term_relations}
    Let $x_1,\dots, x_5$ be five different points on $\mb \P^1$. Then
    $$\sum\limits_{i=1}^5(-1)^iT^2_{c.r.(x_1,\dots,\hat{x_i},\dots, x_5)}=0.$$
\end{corollary}

To prove this lemma, we need another lemma.

\begin{lemma}
\label{lemma:Beilinson_Soule_vanishing}
\begin{enumerate}
    \item Let $f,g\in\k(t)$. We have
        $$[\P^1, f(t)\wedge g(t)]=0.$$
        \item For any $c_3,\dots, c_{m+1}$ we have:
        $$[\P^1, f(t)\wedge g(t)\wedge c_3\wdw c_{m+1}]=0.$$
\end{enumerate}
\end{lemma}

\begin{remark}
    Using Weil reciprocity law it is not difficult to show that the element from the first item is closed. So this item is a manifestation of Beilinson-Soule vanishing $H_{mot}^0(\spec\k, \Q(1))=0$.
\end{remark}

\begin{proof}
We only prove the first item. The proof of the second item is similar. Let $L$ be a finite extension of $\k$ such that the functions $f,g$ can be decomposed into linear factors over $L$. As the morphism $\P^1_L\to \P^1_\k$ is proper and of finite degree we can assume that both $f$ and $g$ are products of linear factors. So it is enough to consider the following three cases
        \begin{enumerate}
        \item Both of the functions $f, g$ are constant. Choose some morphism $\ph \colon\P^1\to \P^1$ of degree $s>1$. We have
        $$[\P^1, f\wedge g]=1/s[\P^1, \ph^*(f\wedge g)]=1/s[\P^1, f\wedge g].$$
        So $[\P^1, f\wedge g]=0$.
        \item $f(t)=c, g(t)=t-a$. We have
            \begin{align*}
                &[\P^1, c\wedge (t-a)]=[\P^1, c\wedge t]=\\
                &[\P^1, c\wedge (1/t)]=-[\P^1, c\wedge t].
            \end{align*}
            So $[\P^1, c\wedge (t-a)]=0$.
            \item $f(t)=t-a, g(t)=t-b$. Let $\ph$ be an automorphism of $\mb P^1$ given by the formula $\ph(t)=a+(b-a)t$. We get
            \begin{align*}
                &[\P^1, (t-a)\wedge (t-b)]=[\P^1, \ph^*((t-a)\wedge (t-b))]=\\&[\P^1, (t(b-a))\wedge (t(b-a)-(b-a))].
            \end{align*}
            It follows from the previous two items that this element is equal to $[\P^1, t\wedge (1-t)]$. Denote by $\ph_2$ an automorphism of $\P^1$ given by the formula $\ph_2(t)=1-t$. We get:
            \begin{align*}
                &[\P^1, t\wedge (1-t)] = [\P^1, \ph_2^*((1-t)\wedge t)]=\\
                &[\P^1, (1-t)\wedge t]=-[\P^1, t\wedge (1-t)].
            \end{align*}
            So this element is zero.
  \end{enumerate}            
\end{proof}

\begin{proof}[The proof of Lemma \ref{lemma:abel_five_term_relation_goes_to_zero}]
    We have:
    \begin{align*}
        &T_{1/a}^2=[\P^1, t\wedge (1-t)\wedge (1-1/(at))]=\\
        &[\P^1, (1/t)\wedge (1-1/t)\wedge (1-t/a)]=-[\P^1, t\wedge (1-t)\wedge (1-a/t)]=-T_{a}^2.
    \end{align*}
    
    The proof of the other formula is similar.
    
    Let us prove the second item. It follows from the previous lemma that for any $\lambda_1,\lambda_2, \lambda_3, \lambda_4\in \k\bs 0$ we have
    $$T(l_1, l_2, l_3, l_4)=T(\lambda_1 l_1, \lambda_2 l_2,\lambda_3 l_3,\lambda_4 l_4).$$
    So we can assume that $V=\k^2$ and $l_1=(1,0), l_2=(-1, 1), l_4=(1,-a), l_4=(0,1)$ for some $a\in \k\bs \{0, 1\}$. We get:
    \begin{align*}
        &T(l_1, l_2, l_3, l_4)=[\P^1, t\wedge (1-t)\wedge (t-a)]=\\
        &[\P^1, t\wedge(1-t)\wedge (1-a/t)]=T^2_{a}.
    \end{align*}
    On the other hand
        $$c.r.(\pi(l_1), \pi(l_2), \pi(l_3), \pi(l_4))=\dfrac{a}{a-1}=\dfrac 1{1-1/a}.$$
Now the statement follows from the first item. The proof of the third item is a direct computation.
\end{proof}

\begin{lemma}
\label{lemma:differential_Totaro_elements}
We have
   \begin{enumerate}
    \item
    \begin{align*}
        &d([\P^1, t\wedge (1-t)\wedge (1-a/t)\wedge c_3\wdw c_m])=\\
        &[\spec \k, a\wedge (1-a)\wedge c_3\wdw c_m].
    \end{align*}
    \item 
    \begin{align*}
        &d([\P^2, x_1\wedge x_2\wedge (1-x_1)\wedge (1-x_2/x_1)\wedge (1-a/x_2)\wedge\\& c_3\wdw c_{m-1}])=
        [\P^1, t\wedge (1-t)\wedge (1-a/t)\wedge a\wedge  c_3\wdw c_{m-1}].
    \end{align*}
\end{enumerate}
\end{lemma}

\begin{proof}\begin{enumerate}
    \item The first formula goes back to Totaro \cite{totaro_1992_milnork}. This formula is obvious as any birational morphism between smooth proper curves is an isomorphism. The only non-trivial terms in the differential corresponds to the point $t=a$.
    \item The differential of the corresponding element in Bloch's higher Chow group was computed in \cite{bloch_rriz_1994_mixed}, see also \cite{goncharov_levin_gangl_2009_muly_polyl_alg_cycles}. Denote by $Z$ the closure of the zeros of the functions $f_i-1$. Let $U=\P^2\bs Z$. It is easy to see that the only irreducible components of the divisors of $f_i$ on $U$ are $x_1=x_2$ and $x_2=a$. Moreover these divisors belong only to the functions $1-x_2/x_1$ and $1-a/x_2$ correspondingly. So we can apply Corollary \ref{cor:differential_exceptional_divisor_zero_Bloch}. For the divisor $x_1=x_2$ the term in the differential is zero. For  the divisor $x_2=a$ we get the element given in the statement of the lemma.
\end{enumerate}
\end{proof}

\begin{proposition}
    The map $\mc{T}_{\ge m-1}(m)$ is a well-defined morphism of complexes.
\end{proposition}
\begin{proof}
We need to check the following three statements:
\begin{enumerate}
    \item Let $x_1,\dots, x_5$ are five different points on $\P^1$ and $c_3,\dots, c_m\in\k\bs\{0\}$. Then
    $$\sum\limits_{i=1}^5(-1)^i[\P^1, t\wedge (1-t)\wedge (1-(c.r.(x_1,\dots,\hat{x_i},\dots , x_5))/t)\wedge c_3\wdw c_m]=0\in \L(\k, m)_{m-1}/\im(d).$$
    \item We have $$[\P^1, t\wedge (1-t)\wedge (1-a/t)\wedge a\wedge c_4\wdw c_m]=0\in \L(\k, m)_{m-1}/\im(d).$$
    \item We have
    \begin{align*}
        &d([\P^1, t\wedge (1-t)\wedge (1-a/t)\wedge c_3\wdw c_m])=\\
        &[\spec \k, a\wedge (1-a)\wedge c_3\wdw c_m].
    \end{align*}
\end{enumerate}
In the case $m=2$ the first statement is Corollary \ref{cor:Abel_five_term_relations}. The proof for an arbitrary $m$ is similar. The second and the third statements follow from Lemma \ref{lemma:differential_Totaro_elements}.
\end{proof}

\subsection{Reduction to algebraically closed field}
By the main result of \cite{rudenko2021strong} the cohomology $$H^i(\tau_{\ge m-1}\Gamma(\spec \k, m))$$ satisfies Galois descent. The cohomology  $$H^i(\tau_{\ge m-1}\Lambda(\spec \k, m))$$ satisfies Galois descent by Corollary \ref{cor:Galois_descent}. Using standard reductions we can assume that the field $\k$ is algebraically closed. So it remains to show that for any algebraically closed field $\k$, the map  $\mc T_{\ge m-1}(m)$ is an isomorphism of complexes.

\subsection{Totaro map is surjective}
\label{subsec:Tatoro_is_surjective}
In this subsection we will prove that the map $\mc T_{\ge m-1}(m)$ is surjective. This follows from the following result:

\begin{proposition}
\label{prop:Lambda_generated_by_Totaro_cycles}
    Let $\k$ be algebraically closed. The group $\L(\k, m)_{m-1}/\im(d)$ is generated by the elements of the form
    $$[\P^1, t\wedge (1-t)\wedge (1-a/t)\wedge c_4\wdw c_{n}].$$   
\end{proposition}

\begin{lemma}
\label{lemma:P1_genrated_by_Totaro}
    Let $S=\P^1\t\P^1$. Consider the element
    $$a=[S, x\wedge (1-x)\wedge \alpha_2\wdw \alpha_k\wedge c_{k+1}\wdw c_n].$$
    In this formula $x$ is the canonical coordinates on the first $\P^1$, $\alpha_i\in\k(\P^1\t\P^1)$ and $c_i\in \k$. For any birational morphism $\ph\colon \wt S\to S$, and any divisor $E$ contracted under $\ph$, the element $[E, \ts_E(\ph^*(a))]$ can be represented as linear combination of the elements of the form
    $$[\P^1, c\wedge (1-c)\wedge  \beta_3\wdw \beta_{k} \wedge c_{k+1}\wdw c_n].$$
    In this formula $c\in \k\bs\{0,1\}$ and $\beta_i\in \k(\P^1)$.
\end{lemma}

\begin{proof}[The proof of Lemma \ref{lemma:P1_genrated_by_Totaro}]
    Denote by $(x,t)$ the canonical coordinates on $S$. Let $\{(x_0,t_0)\}\subset S$ be the image of $E$. If $x_0=0$ then the the restriction of the function $\ph^*(1-x)$  to $E$ is equal to $1$ and so $[E, \ts_E(\ph^*(a))]=0$. The case $x_0=1$ is similar. Let us assume that $x_0=\infty$. In this case, the restriction of the function $\ph^*(x/(1-x))$ to $E$ is equal to $-1$. As the element $a$ can be rewritten in the form 
    $$a=[S, (x/(1-x))\wedge (1-x)\wedge \alpha_2\wdw \alpha_k\wedge c_{k+1}\wdw c_n],$$
    this implies that $[E, \ts_E(\ph^*(a))]=0.$
    
    So we can assume that $x_0\in \k\bs\{0,1\}$. This implies that the restriction of the function $\ph^*(x)$ to $E$ is constant. Denote this constant by $c$. We get
    $$[E, \ts_E(\ph^*(a))]=[E, c\wedge (1-c)\wedge \ts_{E}(\alpha_2\wdw \alpha_k)\wedge c_{k+1}\wdw c_n].$$
    The statement follows.
\end{proof}

\begin{proof}[The proof of Proposition \ref{prop:Lambda_generated_by_Totaro_cycles}]
Denote by $A$ the subgroup of the group $\Lambda(\k, m)_{m-1}/\im(d)$ generated by the elements stated in the lemma.

    By Theorem \ref{th:Lambda_is_generated_by_rationally_connected} we know that the group $\Lambda(\k, m)_{m-1}/\im(d)$ is generated by the elements of the form
    $$[\P^1, (t-a_1)\wdw (t-a_k)\wedge c_{k+1}\wdw c_n].$$
    Let us prove by induction on $k, k\geq 0$, that this element lies in $A$. 
    \begin{enumerate}
        \item The case $k\leq 2$ follows from Lemma \ref{lemma:Beilinson_Soule_vanishing}.
        \item The case $k=3$ follows from Lemma \ref{lemma:abel_five_term_relation_goes_to_zero}.
        \item Let $k\geq 4$. Consider the element
        \begin{align*}
            &[\P^1\t\P^1, x\wedge (1-x)\wedge (x-(t-a_1)/(t-a_2))\wedge\\
            &(t-a_3)\wdw (t-a_k)\wedge c_{k+1}\wdw c_n].
        \end{align*}
    
    It follows from Lemma \ref{lemma:P1_genrated_by_Totaro} that the differential of this element is equal to $x+y$, where $x\in A$ and 
    $$y=(t-a_1)/(t-a_2)\wedge (1-(t-a_1)/(t-a_2))\wedge (t-a_3)\wdw (t-a_{k})\wedge c_{k+1}\wdw c_n.$$
    By inductive assumption this implies that the element
    $$[\P^1, (t-a_1)\wdw (t-a_k)\wedge c_{k+1}\wdw c_n].$$
    lies in $A$.
    \end{enumerate}

\end{proof}

\subsection{Totaro map is injective}
To prove that the map $\mc{CD}_{\ge m-1}(m)$ is injective we construct a map in the opposite direction. 

The following proposition was proved in \cite{goncharov1994polylogarithms}:

\begin{proposition}
\label{prop:tame_symbol}
Let $(F,\nu)$ be a discrete valuation field and $m\geq 1$. There is a unique morphism of complexes $\ts_\nu\colon \Gamma(F,m)\to \Gamma(\ol F_\nu, m-1)[-1]$ satisfying the following conditions:

\begin{enumerate}
   \item For any uniformizer $\pi$ and units $u_2,\dots u_m\in F$ we have $\ts_\nu(\pi\wedge u_2\wedge\dots \wedge u_m)=\overline {u_2}\wedge \dots \wedge \overline{u_m}$.
    \item For any $a\in F\bs \{0,1\}$ with $\nu(a)\ne 0$, an integer $k$ satisfying $2\leq k\leq n$ and any $b\in \Lambda^{n-k}F^\t$ we have $\ts_\nu(\{a\}_k\otimes b)=0$.
    \item For any unit $u$, an integer $k$ satisfying $2\leq k\leq n$ and $b\in \Lambda^{n-k}F^\t$ we have $\ts_\nu(\{u\}_k\otimes b)=-\{\overline u\}_k\otimes \ts_\nu(b)$.
\end{enumerate}
\end{proposition}

Let $X$ be a smooth variety over $\k$ and $D$ be an irreducible divisor on $X$. Denote by $\nu_D$ the discrete valuation corresponding to $D$. We will use the notation $\ts_D$ for the map $\ts_{\nu_D}$.

Let $\k$ be an algebraically closed field and $X$ be a smooth proper curve over $\k$. For any $x\in X(\k)$ we get the map 
$$\ts_x\colon \tau_{\geq m}\Gamma(\k(X),m+1)\to (\tau_{\geq m-1}\Gamma(\k,m))[-1].$$ 
Denote by 
$$Tot_{X}\colon \tau_{\geq m}\Gamma(\k(X),m+1)\to (\tau_{\geq m-1}\Gamma(\k,m))[-1]$$ the sum of these maps over all points $x\in X(\k)$. The following result is a slight generalisation of the main result from \cite{bolbachan_2023_chow}.

\begin{theorem}
\label{th:strong_suslin_reciprocity_law_varities}
    Let $\k$ be an algebraically closed field. To any smooth proper curve $X$ over $\k$ one can associate the canonical map $$\mc H_X\colon \L^{m+1}\k(X)^\t\to \Gamma(\k, m)_{m-1}/\im(\delta_m)$$ satisfying the following properties:
    \begin{enumerate}
        \item The map $\mc H_X$ gives a homotopy between $Tot_{X}$ and the zero map
        \item For any $f_1, f_2\in\k(X)$ and $ c_i\in\k$ we have $$\mc H_X(f_1\wedge f_2\wedge c_3\wdw c_{m+1})=0.$$
        
        \item For any non-constant map $\varphi\colon X\to Y$ we have
        $$\mc H_Y(a)=1/\deg(\varphi)\mc H_X(\varphi^*(a)).$$
        \item $$\mc H_{\P^1}(t\wedge (1-t)\wedge (1-a/t)\wedge c_4\wdw c_{m+1})=-\{a\}_2\otimes c_4\wdw c_{m+1}.$$
        \item Let $S$ be a smooth proper surface. Assume that the element $b\in\L^{m+2}(S)$ is strictly regular (see Definition \ref{def:strictly_regular}). We have
        $$\sum\limits_{C\subset S}\mc H_C(\ts_C(b))=0.$$
    \end{enumerate}
    Moreover the family of the maps $\mc H_X$ are uniquely determined by the properties stated above.
    \end{theorem}

    This theorem will be proved in Section \ref{sec:strong_suslin_reciprocity_law}.

Define a morphism of complexes 
$$\mc{CD}_{\geq m-1}(m)\colon \tau_{\geq m-1}\Lambda(\k,m )\to \tau_{\geq m-1}\Gamma(\k, m)$$
as follows. The image of the element $[X, a]\in \Lambda(\k,m )_{m-1}$ is equal to $-\mc H_{X}(a)$. The image of the element $[\spec \k, a]$ is equal to $a$. It follows from the previous theorem that in this way we get a morphism of complexes. It also follows from the same theorem that $\mc{CD}_{\geq m-1}\circ \mc T_{\geq m-1}(m)=id$. This implies that the map $\mc T_{\ge m-1}(m)$ is injective.

\section{Strong Suslin reciprocity law}
\label{sec:strong_suslin_reciprocity_law}

Denote by $\F_d$ the category of finitely generated extensions of $\k$ of transcendence degree $d$. For $F\in \F_d$, denote by $\dval(F)$ the set of discrete valuations given by an irreducible Cartier divisor on some  birational model of $F$. When $F\in\F_1$ this set is equal to the set of all discrete valuations which are trivial on $\k$. In this case, we denote this set simply by $\val(F)$.

For $F\in\F_1$ denote
$$Tot_F:=\sum\limits_{\nu\in\val(F)}\ts_\nu\colon \tau_{\geq m}\Gamma(F, m+1)\to (\tau_{\geq m-1}\Gamma(\k, m))[-1].$$

If we realise $F$ as $\k(X)$ for some smooth proper curve $X$, then the map $Tot_F$ is identified with the map $Tot_X$ from the previous section.

\begin{definition}
\label{def:SRL}
    Let $F\in \F_1$ be a smooth proper curve over $\k$. 
    \emph{A lifted reciprocity map} on the field $F$  is a map $$h\colon \L^{m+1}(F^\t)\to \Gamma(\k,m)_{m-1}/\im \delta_{m}$$
    such that:
    \begin{enumerate}
    \item The map $h$ gives a homotopy between $Tot_F$ and the zero map:
    \begin{equation}
    \label{diagram:SRL_def}
        \begin{tikzcd}
            (\B_2(F)\otimes \L^{m-1}F^\t)/         \im(\delta_{m+1})\ar[d,"Tot_F"]\ar[r,"\delta_{m+1}"]& \L^{m+1}F^\t \ar[d,"Tot_F"]\ar[dl,"h"]\\
            (\B_2(\k)\otimes \L^{m-2}\k^\t)/\im(\delta_m)\ar[r,"-\delta_{m}"] & \L^{m}\k^\t          
        \end{tikzcd}
    \end{equation}
        \item $h(f_1\wedge f_2\wedge c_3\wdw c_{m+1})=0$
        for any $f_i\in F$ and  $c_i\in\k$.

    \end{enumerate}
\end{definition}

The goal of this section is to prove the following two theorems:

\begin{theorem}
\label{th:strong_suslin_reciprocity_law_field}
    To any field $F\in\F_1$ one can associate a lifted reciprocity map $\mc H_F$ on the field $F$ such that:
    \begin{enumerate}
        \item For any embedding $j\colon F_1\emb F_2$ in $\F_1$ we have $$\mc H_{F_1}(a)=(1/\deg(F_2/F_1))\mc H_{F_2}(j(a)).$$
        \item $$\mc H_{\k(t)}(t\wedge (1-t)\wedge (1-a/t)\wedge c_4\wdw c_{m+1})=-\{a\}_2\otimes c_4\wdw c_{m+1}.$$
    \end{enumerate}
    Moreover the family of the maps $\mc H_F$ are uniquely determined by the properties stated above.
    \end{theorem}

    \begin{theorem}
    \label{th:two_dimensional_reciprocity_law}
        \begin{enumerate}
            \item For any field $L\in\F_2$ and any $b\in\L^{m+2}(L^\t)$ we have
$$\sum\limits_{\nu\in\dval(L)}\mc H_{\ol L_\nu}(\ts_\nu(b))=0.$$
           \item Let $S$ be a smooth proper surface over $\k$ and $b\in\Lambda^{m+2} k(S)^\t$, such that $b$ is strictly regular. Then

            $$\sum\limits_{C\subset S}\mc H_{\k(C)}(\ts_C(b))=0.$$
        \end{enumerate}
        
    \end{theorem}

Theorem \ref{th:strong_suslin_reciprocity_law_varities} directly follows from Theorem \ref{th:strong_suslin_reciprocity_law_field} and Theorem \ref{th:two_dimensional_reciprocity_law}.

\subsection{Case of of the field $\k(t)$}
First of all we need the following lemma:
\begin{lemma}
\label{lemma:Res_is_zero}
Let $a\in \Gamma(\k(t),m+1)_{m}$. If $\delta_{m+1}(a)$ lies in the image of the multiplication map $\Lambda^2 \k(t)^\t\otimes \Lambda^{m-1}\k^\t\to\Lambda^{m+1}\k(t)^\t$, then we have:
$$\sum\limits_{\nu\in\val(\k(t))}\ts_\nu(a)=0\in \Gamma(\k, m)_{m-1}/\im(d).$$
\end{lemma}
\begin{proof}
\begin{enumerate}
    \item Let $b=\delta_{m+1}(a)$. Subtracting a linear combination of some elements of the form $\{(t-a)/(t-b)\}_2\otimes c_3\wedge\dots\wedge c_{m+1}, a,b,c_i\in \k$ we can assume that $b$ lies in the image of the multiplication map $\k(t)^\t\otimes \Lambda^{m}\k^\t\to\Lambda^{m+1}\k(t)^\t$.
    \item Let $a''=(-1)^m\ts_{\infty}(a\wedge (1/t))$. For $x\in\k$ denote $a'_x= (-1)^{m+1}\ts_{x}(a)\wedge (t-x)$. Since the total residue of the elements $a'',a'_x$ is equal to zero, it is enough to prove the statement for the element $$\wt a=a-a''-\sum\limits_{x\in \k}a'_x.$$
    \item Let $\wt b=\delta_{m+1}(\wt a)$. We claim that $\wt b=0$. Indeed, the element $\wt b$ can be represented in the form
    $$\wt b=b'+\sum\limits_{x\in\k}(t-x)\wedge b_x,$$
    
    where $b'\in\Lambda^{m+1}\k^\t, b_x\in\Lambda^m\k^\t$.
    Simple computation shows that for any $x$ we have $\ts_{x}(b)=\ts_{\infty}(b\wedge (1/t))=0$. It follows that $b'=b_x=0$. So $\wt b=0$.
    \item So we can assume that $\delta_{m+1}(a)=0$. In this case the statement follows from \cite[Corollary 1.4]{rudenko2021strong}.
\end{enumerate}

\end{proof}
\begin{proposition}
\label{prop:strong_low_P^1}
\label{prop:SRL_P1}
On the field $\k(t)$ there is a unique lifted reciprocity map. 
\end{proposition}

We will denote this lifted reciprocity map by $h_{\k(t)}$.

\begin{proof}Denote by $A_1\subset \Lambda^{m+1}\k(t)^\t$ the image of $\delta_{m+1}$ and by $A_2$ the image of the multiplication map $\Lambda^2 \k(t)^\t\otimes \Lambda^{m-1}\k^\t\to\Lambda^{m+1}\k(t)^\t$. Elementary calculation shows that $A_1$ and $A_2$ together generate $\Lambda^{m+1}\k(t)^\t$. As any lifted reciprocity map is uniquely determined on $A_1$ and on $A_2$, uniqueness follows.

To show existence, define a map $$h\colon \Lambda^{m+1}\k(t)^\t\to (B_2(\k)\otimes \Lambda^{m-2}\k^\t)/\im(\delta_m)$$ as follows.

Let $a=a_1+a_2$, where $a_1\in A_1, a_2\in A_2$. Choose some $b\in \delta_{m+1}^{-1}(a_1)$ and define $h(a)=\sum\limits_{x\in\mathbb P^1}\ts_x(b)$. This map is well-defined by Lemma \ref{lemma:Res_is_zero}. 

We already know that $h$ satisfies the second property of Definition \ref{def:SRL}  and that the upper-left triangle of diagram (\ref{diagram:SRL_def}) is commutative. Let us show the bottom-right triangle is commutative. On $A_2$ its commutativity follows from Weil reciprocity law. Its commutativity on $A_1$ follows from the fact that total residue map $Tot_F$ is a morphism of complexes.
\end{proof}

\begin{corollary}
\label{cor:SRL_on_Totaro}
    $$h_{\k(t)}(t\wedge (1-t)\wedge (1-a/t)\wedge c_4\wdw c_{m+1})=-\{a\}_2\otimes c_4\wdw c_{m+1}.$$
\end{corollary}
\begin{proof}
Let $$b=\{t\}_2\otimes (1-a/t)\wedge c_4\wdw c_{m+1}.$$

By the definition of the map $h_{\k(t)}$ we know that
$Tot_{k(t)}(b)=h_{k(t)}\delta_{m+1}(b).$ We have
$$Tot_{k(t)}(b)=-\{a\}_2\otimes c_4\wdw c_{m+1}, \delta_{m+1}(b)=t\wedge (1-t)\wedge (1-a/t)\wedge c_4\wdw c_{m+1}.$$
The statement follows.
\end{proof}

\subsection{The construction of the lift}
\label{sec:prel_results:lift}

\begin{definition}
\label{def:types_of_valuations}
Let $F\in \F_1$. A valuation $\nu\in \dval(F(t))$ is called \emph{general} if it corresponds to some irreducible polynomial over $F$. The set of general valuations are in bijection with  the set of all closed points on the affine line over $F$, which we denote by $\mb A_{F,(0)}^1$.  A valuation is called \emph{special} if it is not general. Denote the set of general (resp. special) valuations by $\dval(F(t))_{gen}$ (resp. $\dval(F(t))_{sp}$). 
\end{definition}

\begin{remark}
\label{rem:types_of_valuations}

Let $F\in\F_1$. Let us realize $F$ as a field of fractions on some smooth projective curve $X$ over $k$. Set $S=X\t \mathbb P^1$. It can be checked that a valuation $\nu\in \dval(\k(S))$ is special in the following two cases:

\begin{enumerate}
    \item There is a birational morphism $p\colon \wt S\to S$, and the valuation $\nu$ corresponds to some irreducible divisor $D\subset \wt S$ contracted under $p$.
    \item The valuation $\nu$ corresponds to some of the divisors $X\t\{\infty\}, \{a\}\t\mathbb P^1, a\in X$.
\end{enumerate}

Otherwise, the valuation $\nu$ is general.
It follows from this description that if $\nu$ is a special valuation different from $X\times \{\infty\}$, then the residue field $\overline{F(t)}_\nu$ is isomorphic to $\k(t)$.
\end{remark}

\begin{definition}
\label{def:lift}
Let $F\in \F_1$, $j, m\in \mathbb N$, $\nu\in \dval(F(t))_{gen}$ and $a\in \Gamma(\ol {F(t)}_\nu, m)_j$. \emph{A lift} of the element $a$ is an element $b\in \Gamma(F(t), m+1)_{j+1}$ satisfying the following two properties:
\begin{enumerate}
    \item $\ts_\nu(b)=a$ and
    \item for any general valuation $\nu'\in\dval(F(t))_{gen}$ different from $\nu$, we have $\ts_{\nu'}(b)=0$.
\end{enumerate}

The set of all lifts of the element $a$ is denoted by $\mathcal L(a)$.
\end{definition}

\begin{theorem}
\label{th:main_exact_sequence}
Let $F\in \F_1$ and $j\in\{m,m+1 \}$. For any $\nu\in \dval(F(t))_{gen}$ and $a\in \Gamma(\ol{F(t)}_\nu, m+1)_j$, the following statements hold:
\begin{enumerate}
    \item  The set $\mathcal L(a)$ is non-empty.
    \item Let us assume that $j=m+1$. For any $b_1,b_2\in \mathcal L(a)$, the element $b_1-b_2$ can be represented in the form $a_1+\delta_{m+2}(a_2)$, where $a_1\in \Lambda^{m+2} F^\t$ and $a_2\in \Gamma(F(t),m+2)_{m+1}$ such that for any $\nu\in\dval(F(t))_{gen}$ the element $\ts_\nu(a_2)$ lies in the image of the map $\delta_{m+1}$.
\end{enumerate}
\end{theorem}

The proof of this theorem relies on Lemma 2.8 and Proposition 2.9 from loc.cit. Both of these statements were formulated for arbitrary $m$. The proof of Theorem 2.6. from loc.cit. is quite formal and does not use any specific properties of the case $m=2$ and can be easily generalised to arbitrary $m$.

\subsection{Parshin reciprocity law}
\label{sec:prel_results:Parshin}

We need the following statement:

\begin{theorem}
\label{th:Parshin_sum_point_curve}
Let $L\in\F_2$ and $j\in\{m+1, m+2\}$. For any $b\in \Gamma(L,m+2)_j$ and all but finitely many $\mu\in \dval(L)$ the following sum is zero:
$$\sum\limits_{\mu'\in \val(\ol L_\mu)}\ts_{\mu'}\ts_\mu(b)=0.$$
Moreover the following sum is zero:
$$\sum\limits_{\mu\in \dval(L)}\sum\limits_{\mu'\in \val(\oL_\mu)}\ts_{\mu'}\ts_\mu(b)=0.$$
\end{theorem}

The proof of this theorem is completely similar to the proof of Theorem 2.10 from \cite{bolbachan_2023_chow}. 

\begin{remark}
    The proof of Theorem 2.10 from loc.cit relies on Lemma 2.14. The proof of this lemma is not correct. However this lemma can be proved similarly to the proof Lemma \ref{lemma:char_of_strictly_regular} from this paper. 
\end{remark}

Also, we need the following lemma:

\begin{lemma}
\label{lemma:finitnes_of_sum}
Let $S$ be a smooth proper surface. Let $b\in \Lambda^{m+2} L^\t$ be strictly regular. For any birational morphism $\ph\colon \wt S\to S$ and any divisor $E\subset \wt S$ contracted under $\ph$, the element $\ts_E(\ph^*(b))$ lies in the image of the multiplication map $\overline L_\nu^\t\otimes \Lambda^m \k^\t\to \Lambda^{m+1}\overline L_\nu^\t$.
\end{lemma}

\begin{corollary}
\label{cor:finitnes_of_sum}
    Let $L\in\F_2$. For any $b\in \Lambda^{m+2} L^\t$ and all but a finite number of $\nu\in \dval(L)$ the element $\ts_\nu(b)$ belongs to the image of the multiplication map $\overline L_\nu^\t\otimes \Lambda^m \k^\t\to \Lambda^{m+1}\overline L_\nu^\t$.
\end{corollary}

\subsection{The proof of Theorem \ref{th:strong_suslin_reciprocity_law_field}}

Denote by $\Set$ the category of sets. 
Define a contravariant functor $$\SRL\colon \F_1\to \Set$$ as follows. For any $F\in \F_1$ the set $\SRL(F)$ is equal to the set of all lifted reciprocity maps on $F$. If $j\colon F_1\emb F_2$ then $\SRL(j)(h_{F_2})=h_{F_1}$,where $h_{F_1}(\alpha_1\wdw \alpha_{m+1}):=\dfrac 1{\deg j}h_{F_2}(j(\alpha_1)\wdw j(\alpha_{m+1}))$. The fact that in this way we get a functor can be proved similarly to \cite[Proposition 2.1]{bolbachan_2023_chow}.

Let $F\in \F_1$, $\nu\in\dval(F(t))$. Denote the field $F(t)$ by $L$. Our goal is to construct the map $\N_\nu\colon \SRL(F)\to \SRL(\oL_\nu)$. We will do this in the following three steps:

\begin{enumerate}
    \item We will define this map when $\nu$ is a special valuation.
    \item Using the construction of the lift from Section \ref{sec:prel_results:lift}, for any general valuation $\nu\in\dval(L)$, we will define a map $$\N_\nu\colon \SRL(F)\to Hom(\Gamma(\ol L_\nu, m+1)_{m+1}\to \Gamma(\k,m)_{m-1}/\im(\delta_m)).$$
    \item Using Theorem \ref{th:Parshin_sum_point_curve} from Section \ref{sec:prel_results:Parshin}, we will show that for any $h\in\SRL(F)$, the map $\N_\nu(h)$ is a lifted reciprocity map on the field $\oL_\nu$. So, $\N_\nu$ gives a map $\SRL(F)\to \SRL(\ol{L}_\nu)$.
\end{enumerate}

Denote by $\nu_{\infty, F}\in\dval(F(t))$ the discrete valuation corresponding to the point $\infty\in\P^1_F$. Let $\nu$ be special. If $\nu=\nu_{\infty,F}$ then define $\N_\nu(h)=h$ (here we have used the identification of $\oL_{\nu_{\infty,F}}$ with $F$). In the other case we have $\overline{L}_\nu\simeq \k(t)$ (see Remark \ref{rem:types_of_valuations}). In this case define $\N_\nu(h)$ to be the unique lifted reciprocity map from Proposition \ref{prop:SRL_P1}. We have defined $\N_\nu$ for any $\nu\in \dval(L)_{sp}$. 

Let $h\in\SRL(F)$. Define a map $H_h\colon \Lambda^{m+2} L^\t\to \Gamma(\k,m)_{m-1}/\im(\delta_m)$ by the following formula:
$$H_h(b)=-\sum\limits_{\mu\in \dval(L)_{sp}}\mc N_\mu(h)(\ts_{\mu}(b)).$$ 
This sum is well-defined by Corollary \ref{cor:finitnes_of_sum}.
\begin{definition}
\label{def:theta_gen}
Let $\nu\in\dval(L)_{gen}$. Define a map $$\mc N_\nu\colon \SRL(F)\to \Hom(\Lambda^{m+1}\ol L_\nu^\t,  \Gamma(\k,m)_{m-1}/\im(\delta_m))$$ as follows. Let $h\in\SRL(F)$ and $a\in \Lambda^{m+1}\ol L_\nu^\t$. Choose some lift $b\in\mc L(a)$ and define the element $\mc N_\nu(h)(a)$ by the formula $H_h(b)$. 
\end{definition}

Similarly to Section 3.1 of loc.cit. it can be shown that this definition is well-defined and gives a map $\mc N_\nu\colon \SRL(F)\to \SRL(\ol L_\nu).$ We omit details.

\begin{definition}
\label{def:norm_map}
Let $j\colon F_1\emb F_2$ be an extension of some fields from $\F_1$. Let $a$ be some generator of $F_2$ over $F_1$. Denote by $p_a\in F[t]$ the minimal polynomial of $a$ over $F_1$. Denote by $\nu_a$ the corresponding valuation. The residue field $\overline{L}_{\nu_{a}}$ is canonically isomorphic to $F_2$. So we get a map $\N_{{\nu_a}}\colon\SRL(F_1)\to \SRL(F_2)$, which we denote by $N_{F_2/F_1, a}$. This map is called \emph{the norm map}. 
\end{definition}

Then one can prove the following theorem:

\begin{theorem}
    The map $N_{F_2/F_1, a}$ does not depend on $a$. Denote it simply by $N_{F_2/F_1}$. We have:
    \begin{enumerate}
        \item If $F_1\subset F_2\subset F_3$ are extensions from $\F_1$, then $$N_{F_3/F_2}\circ N_{F_2/F_1}=N_{F_3/F_1}.$$
        \item Let $j\colon F_1\emb F_2$. We have $\SRL(j)\circ N_{F_2/F_1}=id$.
        \item Let $F\in\F_1$. Choose some embedding $j\colon \k(t)\emb F$. The element $$N_{F/\k(t)}(h_{\k(t)})$$ does not depend on $j$.
    \end{enumerate}
\end{theorem}

Let us prove Theorem \ref{th:strong_suslin_reciprocity_law_field}. Item 1 follows directly  from the previous theorem. (See the proof of Theorem 1.8. from loc. cit.). Item 2 follows from Corollary \ref{cor:SRL_on_Totaro} and Proposition \ref{prop:strong_low_P^1}.

Let us prove Theorem \ref{th:two_dimensional_reciprocity_law}. The proof of the first item is similar to the proof of Corollary 1.13 from loc.cit. The second item follows from Lemma \ref{lemma:finitnes_of_sum}.

\bibliographystyle{alpha}      
\bibliography{mylib} 
\end{document}